\documentclass[a4paper,11pt,reqno]{amsart}

\usepackage{
amsfonts,
amsmath,
amsopn,
amssymb,
amsthm,
bbm,
bbold,
dsfont,
enumitem,
graphicx,
mathrsfs,
mathtools,
soul,
subfig,
verbatim,
xcolor,
xspace
}

\usepackage{geometry}
\usepackage[hidelinks]{hyperref}
\usepackage[utf8]{inputenc}

\mathtoolsset{showonlyrefs}

\newtheorem{theorem}{Theorem}[section]
\newtheorem{lemma}[theorem]{Lemma}
\newtheorem{proposition}[theorem]{Proposition}
\newtheorem{corollary}[theorem]{Corollary}

\theoremstyle{remark}
\newtheorem{remark}[theorem]{Remark}
\newtheorem*{remark*}{Remark}

\theoremstyle{definition}
\newtheorem{definition}[theorem]{Definition}

\newcommand{\R}{\mathbb{R}}
\newcommand{\Rd}{\mathbb{R}^2}

\newcommand{\C}{\mathbb{C}}

\newcommand{\N}{\mathbb{N}}

\newcommand{\Eps}{\mathcal{E}}
\newcommand{\G}{\mathcal{G}}

\newcommand{\F}{\mathcal{F}}

\renewcommand{\leq}{\leqslant}
\renewcommand{\geq}{\geqslant}

\newcommand{\la}{\lambda}

\newcommand{\al}{\alpha}

\newcommand{\ga}{\gamma}
\newcommand{\ep}{\varepsilon}

\newcommand{\x}{\mathbf{x}}

\renewcommand{\k}{\mathbf{k}}
\newcommand{\z}{\mathbf{0}}

\newcommand{\ds}{\,ds}
\newcommand{\dt}{\,dt}
\newcommand{\dx}{\,d\x}

\newcommand{\lap}{\Delta}
\newcommand{\na}{\nabla}

\newcommand{\f}[2]{\frac{#1}{#2}}

\newcommand{\deb}{\rightharpoonup}

\title[Ground states for the planar NLSE with a point defect]{Ground states for the planar NLSE with a point defect as minimizers of the constrained energy} 

\author[R. Adami]{Riccardo Adami}
\address{Politecnico di Torino, Dipartimento di Scienze Matematiche ``G.L. Lagrange'',Corso Duca degli Abruzzi, 24, 10129, Torino, Italy}

\email{riccardo.adami@polito.it}

\author[F. Boni]{Filippo Boni}
\address{Università degli Studi di Napoli Federico II,Dipartimento di Matematica e Applicazioni ``Renato Caccioppoli”, Via Cintia, Monte S. Angelo, 80126, Napoli, Italy}

\email{filippo.boni@unina.it}

\author[R. Carlone]{Raffaele Carlone}
\address{Università degli Studi di Napoli Federico II,Dipartimento di Matematica e Applicazioni ``Renato Caccioppoli”, Via Cintia, Monte S. Angelo, 80126, Napoli, Italy}
\email{raffaele.carlone@unina.it}

\author[L. Tentarelli]{Lorenzo Tentarelli}
\address{Politecnico di Torino, Dipartimento di Scienze Matematiche ``G.L. Lagrange'',Corso Duca degli Abruzzi, 24, 10129, Torino, Italy}
\email{lorenzo.tentarelli@polito.it}

\date{\today}

\begin{document}

%%%%%%%%%%%%%%%%%%%%%%%%%%%%%%%%%%
%%%%%%%%%% Intestazione %%%%%%%%%%
%%%%%%%%%%%%%%%%%%%%%%%%%%%%%%%%%%

\begin{abstract}
We investigate the ground states for the focusing, subcritical nonlinear Schr\"odinger equation with a point defect in dimension two, defined as the minimizers of the energy functional at fixed mass. We prove that ground states exist for every positive mass and show a logarithmic singularity at the defect. Moreover, up to a multiplication by a constant phase, they are positive, radially symmetric, and decreasing along the radial direction.
In order to overcome the obstacles arising from the uncommon structure of the energy space, that complicates the application of standard rearrangement theory, we move to the study of the minimizers of the action functional on the Nehari manifold and then establish a connection with the original problem. A refinement of a classical result on rearrangements is proved to obtain qualitative features of the ground states.
\end{abstract}

\maketitle

\vspace{-.5cm}

\noindent {\footnotesize \textul{Keywords:} standing waves, nonlinear Schr\"odinger, ground states, delta interaction, radially symmetric solutions, rearrangements.}

%%%%%%%%%%%%%%%%%%%%%%%%%%%%%%%%%%%%%%%%%%%%%%%%%%%%%%%%%%%%%%%%%%%%%%%%%%%%%%%%
%%%%%%%%%%%%%%%%%%%%%%%%%%%%%%%%%%%%%%%%%%%%%%%%%%%%%%%%%%%%%%%%%%%%%%%%%%%%%%%%
%%%%%%%%%%%%%%%%%%%%%%%%%%%%%%%%%%%%%%%%%%%%%%%%%%%%%%%%%%%%%%%%%%%%%%%%%%%%%%%%
%%%%%%%%%%%%%%%%%%%%%%%%%%%%%%%%%%%%%%%%%%%%%%%%%%%%%%%%%%%%%%%%%%%%%%%%%%%%%%%%
%%%%%%%%%%%%%%%%%%%%%%%%%%%%%%%%%%%%%%%%%%%%%%%%%%%%%%%%%%%%%%%%%%%%%%%%%%%%%%%%
%%%%%%%%%%%%%%%%%%%%%%%%%%%%%%%%%%%%%%%%%%%%%%%%%%%%%%%%%%%%%%%%%%%%%%%%%%%%%%%%
%%%%%%%%%%%%%%%%%%%%%%%%%%%%%%%%%%%%%%%%%%%%%%%%%%%%%%%%%%%%%%%%%%%%%%%%%%%%%%%%
%%%%%%%%%%%%%%%%%%%%%%%%%%%%%%%%%%%%%%%%%%%%%%%%%%%%%%%%%%%%%%%%%%%%%%%%%%%%%%%%
%%%%%%%%%%%%%%%%%%%%%%%%%%%%%%%%%%%%%%%%%%%%%%%%%%%%%%%%%%%%%%%%%%%%%%%%%%%%%%%%
%%%%%%%%%%%%%%%%%%%%%%%%%%%%%%%%%%%%%%%%%%%%%%%%%%%%%%%%%%%%%%%%%%%%%%%%%%%%%%%%
%%%%%%%%%%%%%%%%%%%%%%%%%%%%%%%%%%%%%%%%%%%%%%%%%%%%%%%%%%%%%%%%%%%%%%%%%%%%%%%%
%%%%%%%%%%%%%%%%%%%%%%%%%%%%%%%%%%%%%%%%%%%%%%%%%%%%%%%%%%%%%%%%%%%%%%%%%%%%%%%%

%%%%%%%%%%%%%%%%%%%%%%%%%%%%%%%%%%
%%%%%%%%%% Introduzione %%%%%%%%%%
%%%%%%%%%%%%%%%%%%%%%%%%%%%%%%%%%%

\section{Introduction}

The Nonlinear Schr\"odinger Equation (NLSE) has provided for almost fifty years the effective description of the evolution of the wave function of 
a Bose-Einstein condensate (BEC) in the Gross-Pitaevskii regime.
More recently, interest has grown in the possibility of modeling a BEC in the presence of defects or impurities by means of a NLSE with an additional pointwise interaction located at the defect. If the spatial scale of the impurity is supposed to be much smaller than the dispersion of the wave function, one can describe it  by means of a  Dirac's delta potential \cite{CC-94,SM-20,SCMS-20}, obtaining the evolution equation
% \emph{point-defected} NLSE ($\delta$-NLSE), i.e.
\begin{equation}
 \label{eq-tNLS_form}
 i\f{\partial\psi}{\partial t}=(-\Delta+\alpha\delta_0)\psi +\beta |\psi|^{p-2}\psi,\qquad\alpha\in\R\setminus\{0\},\quad\beta\in \R\setminus\{0\}, \quad p>2,
\end{equation}
where the sign of $\beta$ establishes the focusing or defocusing character of the equation, and, correspondingly, the attractive or repulsive behaviour of the condensate.

A large part of the available results concerns the one-dimensional case, that models the so-called cigar-shaped condensates.
In particular, well-posedness was established in \cite{AN-09} for the entire class of pointwise potentials, while existence and stability of standing waves were shown in \cite{FOO,FJ,LFFKS-08, AN-09, AN-CMP13, ANV-DCDS13,ANR-20}. On the other hand, the well-posedness for the two and three-dimensional models was established in \cite{CFN-21}. Here we aim at discussing the existence and the properties of ground states for the two-dimensional case in the focusing regime, i.e. when $\beta<0$.

In fact, equation \eqref{eq-tNLS_form} is just formal in dimension two. In order to state it rigorously, one has to replace $-\Delta+\alpha\delta_0$ with a suitable self-adjoint operator $H_{\alpha}$, 
acting on $L^{2}(\R^{2})$ (see Section \ref{subsec-main} below). Such operator acts as the Laplacian far from the origin and its domain contains functions that exhibit a logarithmic singularity at the origin, like the fundamental solution of the Laplacian. As shown in \cite{AGHKH-88}, the operator $H_{\alpha}$ can be also understood as the limit of a sequence of Schr\"odinger operators $-\Delta+ V_{\ep}$, where, for every $\ep > 0$, the potential $V_{\ep}$ is regular, peaked, shrinking around the origin as $\ep\to 0$, and suitably renormalized: this is expected from an operator that aims at embodying a delta interaction at the origin.
 
Incidentally, let us mention that the literature on the NLSE with a potential is much wider than the corresponding one about NLSE with singular potentials: among the others, we mention the seminal works \cite{FW-86,RW-88} and the papers \cite{FO-03, FO-03-I} for their results about stability and instability of standing waves.

Moreover, the analysis of models with point interactions like \eqref{eq-tNLS_form} is strictly connected with the study of singular solutions for elliptic equations, that traces back to the eighties \cite{BL-81,CC-94, GKS-20, JPY-94, L-80, NS-86, NS-86bis, V-81}. In particular, it is well-known that solutions to the focusing stationary NLSE
\begin{equation}
\label{nls-stat}
-\Delta u - |u|^{p-2}u-\omega u=0,
\end{equation}
that are regular in $\R^{2}\setminus \{0\}$, vanish at infinity and are singular at $0$, behave like the fundamental solution of the Laplacian at the origin. 
%In this regard, the definition of the domain of $H_{\alpha}$ proves to be coherent.

Finally, we highlight that all results and proofs contained in the present paper concern the space dimension two. In \cite{ABCT-3d} we extend the results to the three-dimensional case.

\subsection{Future developments}
In our intention the present paper is the first step of a research programme devoted to the study of the standing waves of the NLSE on multi-dimensional structures, that is domains consisting of pieces of different space dimensions glued together through suitable boundary conditions. Such structures are known in the literature as \emph{quantum hybrids} and one of the simplest models is provided by a plane attached to a half-line.
It has been shown \cite{CE-11,ES-87,ES-88} that the conditions to be imposed at the junction between the plane and the half-line prescribe a logarithmic singularity for the restriction of the wave function to the plane, exactly like for the Schr\"odinger equation with a point interaction. Therefore, the present work lays the foundations of our research plan towards nonlinear quantum hybrids.

A further branch of the same research project  concerns a different family of singular perturbation of the Laplacian, called concentrated nonlinearities, namely
\begin{equation}
 \label{eq-tNLS_conc}
 i\f{\partial\psi}{\partial t}=(-\Delta+\tau|\psi|^{p-2}\delta_0)\psi,\qquad\tau\in\R\setminus\{0\},\quad p>2.
\end{equation}
The standard nonlinearity is no longer there, while a pointwise nonlinearity is present at the defect. Specifically, this can be done by replacing the strength of the delta interaction $\alpha$ by a nonlinear term that depends on the solution. As a particular choice, we took $\tau|\psi|^{p-2}$. Such equation has been studied in one \cite{AFH-21,AT-JFA01,CFT-Non19,HL-20,HL-21}, two \cite{ACCT-20,ACCT-21,CCT-ANIHPC19} and three dimensions \cite{ADFT-ANIHPC03,ADFT-ANIHPC04,ANO-JMP13,ANO-DCDS16}. It is, then, natural to investigate the dynamics of a system in the presence of both types of nonlinearity. The one-dimensional case and the case of the star graphs have been already addressed in \cite{BD-21} and \cite{ABD-20}, respectively, while high-dimensional cases are still unexplored.

%%%%%%%%%%%%%%%%%%%%%%%%%%%%%%%%%%%%%%%%%%%%%%%%%%%%%%%%%%%%%%%%%%%%%%%%%%%%%%%%
%%%%%%%%%%%%%%%%%%%%%%%%%%%%%%%%%%%%%%%%%%%%%%%%%%%%%%%%%%%%%%%%%%%%%%%%%%%%%%%%
%%%%%%%%%%%%%%%%%%%%%%%%%%%%%%%%%%%%%%%%%%%%%%%%%%%%%%%%%%%%%%%%%%%%%%%%%%%%%%%%
%%%%%%%%%%%%%%%%%%%%%%%%%%%%%%%%%%%%%%%%%%%%%%%%%%%%%%%%%%%%%%%%%%%%%%%%%%%%%%%%
%%%%%%%%%%%%%%%%%%%%%%%%%%%%%%%%%%%%%%%%%%%%%%%%%%%%%%%%%%%%%%%%%%%%%%%%%%%%%%%%
%%%%%%%%%%%%%%%%%%%%%%%%%%%%%%%%%%%%%%%%%%%%%%%%%%%%%%%%%%%%%%%%%%%%%%%%%%%%%%%%

\subsection{Setting and main results}
\label{subsec-main}

Let us stress again that writing \eqref{eq-tNLS_form} is formal, as in two dimensions the delta interaction is not a small perturbation of the Laplacian. It is well-known that a rigorous definition can be given through the theory of  self-adjoint extensions of hermitian operators. Eventually, one finds  \cite{AGHKH-88} that there exists a family $(H_\alpha)_{\alpha\in\R}$ of self-adjoint operators that realizes a nontrivial point perturbation of $-\Delta$.
 
The domains of such operators are
\begin{multline}
\label{domop}
D(H_{\alpha}):=
\big\{v\in L^{2}(\Rd):\exists q\in\C,\,\lambda>0\:\text{ s.t. }\: \\ v-q\G_\la=:\phi_{\la}\in H^{2}(\Rd)\:\text{ and }\: \phi_{\la}(0)=\left(\alpha+\theta_\la\right)q\big\},
\end{multline}
and the action reads
\begin{equation}
\label{eq-actH}
H_{\alpha}v:=-\lap\phi_{\la}-q\la\G_\la,\qquad\forall v\in D(H_{\alpha}).
\end{equation}
We denoted
\begin{equation}
\label{eq-thla}
\theta_{\la}:=\f{\log\left(\f{\sqrt{\la}}{2}\right)+\ga}{2\pi},
\end{equation}
with $\gamma$ denoting the Euler-Mascheroni constant, and $\G_\la$ the Green's function of $-\lap+\la$, namely
\begin{equation}
 \label{eq-green}
 \G_\la(\x):=\f{1}{2\pi}\F^{-1}\big[(|\k|^2+\la)^{-1}\big](\x)=\f{K_0(\sqrt{\la}\x)}{2\pi}.
\end{equation}
Here, $K_0$ denotes the modified Bessel function of the second kind of order $0$, also known as Macdonald function (see \cite[Section 9.6]{AS-65}), and $\F$ denotes the unitary Fourier transform. The function $\G_\lambda$ has a singular behaviour at the origin, namely
$$ \G_\lambda (\x) = - \frac {\log |\x|}{2 \pi} + o ( \log |\x| ), \qquad \x \to 0 $$
that prevents $\G_\lambda$ from belonging to $H^1 (\Rd)$.

Functions in $D(H_{\alpha})$ consist then of a \emph{regular part} $\phi_\la$, on which the operator acts as the standard Laplacian, and a \emph{singular part} $q\G_\la$, on which the operator acts as the multiplication by $-\lambda$. The two components are connected by the {boundary condition} $\phi_{\la}(0)=\left(\alpha+\theta_\la\right)q$. The strength $q$ of the singular part is called the {\em charge}. We highlight that $\la$ is a dumb parameter as it does not affect neither the definition of $H_\alpha$ nor the charge $q$ (see Remark \ref{qindla}). 

Finally, we recall that the spectrum of $H_\alpha$ is given by
\begin{equation}
 \label{eq-spectrum}
 \sigma(H_\alpha)=\{\ell_\alpha\}\cup[0,+\infty),\,\,\text{with}\,\,\,\ell_\alpha:=-4e^{-4\pi\alpha-2\gamma}<0\,\,\,\text{sole eigenvalue for any } \alpha \in \R.
\end{equation}

\medskip
The rigorous form of the focusing NLSE with a pointwise impurity ($\delta$-NLSE) is therefore 
\begin{equation}
\label{tdNLS}
i\f{\partial \psi}{\partial t}=H_\alpha \psi-|\psi|^{p-2}\psi,\qquad\alpha\in\R,\quad p>2.
\end{equation}
As proven in \cite{CFN-21}, the flow generated by \eqref{tdNLS} preserves the mass.

\begin{remark}
In getting \eqref{tdNLS} from \eqref{eq-tNLS_form}, we fixed $\beta=-1$. No generality is lost since, given any solution $\psi$ of \eqref{tdNLS}, then $\psi_{\beta}:=(-\beta)^{-\f{1}{p-2}}\psi$ solves 
\begin{equation*}
i\f{\partial \psi_{\beta}}{\partial t}=H_\alpha \psi_{\beta}+\beta|\psi_{\beta}|^{p-2}\psi_{\beta},\qquad\alpha\in\R, \quad \beta<0.
\end{equation*}
In doing this, many relevant thresholds of the equation could a priori be affected, but this is not actually the case since no threshold appears in the main results of the paper.
%and viceversa.
\end{remark}

Standing waves are solutions to \eqref{tdNLS} of the form $\psi=e^{i\omega t}u(\x)$, with $\omega\in\R$. An easy computation shows that $\psi$ is a standing wave for \eqref{tdNLS} whenever $u$ is a bound state, i.e. 
\begin{gather}
 \label{eq-regbound} u\in D(H_\alpha),\\
 \label{EL3} H_{\alpha}u+\omega u-|u|^{p-2}u=0.
\end{gather}
Among all the bound states, we focus on ground states. The proof that a ground state satisfies \eqref{eq-regbound} and \eqref{EL3} is straightforward and is reported in Appendix \ref{app-gbstates}.

In order to give a precise definition of the ground states of \eqref{EL3}, we first introduce the quadratic form associated with $H_\alpha$, which has domain
\begin{equation}
\label{dom}
D:=\big\{v\in L^{2}(\Rd):\exists q\in\C,\,\lambda>0\:\text{ s.t. }\:v-q\G_\la=:\phi_{\la}\in H^{1}(\Rd)\big\},
\end{equation}
and action
\begin{equation}
\label{eq-Q}
Q(v):=\|\na\phi_\la\|^2_{2}+\la\big(\|\phi_\la\|^2_{2}-\|v\|^2_{2}\big)+\left(\alpha+\theta_\la\right)|q|^2,\qquad\forall v\in D,
\end{equation}
where we denoted by $\left\langle \cdot,\cdot\right\rangle$ the hermitian product in $L^2(\Rd)$ and by $\|\cdot\|_p$ the usual norm in  $L^p(\Rd)$. As \eqref{eq-Q} is the quadratic form of the self-adjoint operator $H_\alpha$, it is independent of the choice of $\lambda$. Notice that, as expected when passing from operator to quadratic form, functions in $D$ have a rougher regular part than functions in $D(H_\alpha)$ (from $H^2(\Rd)$ to $H^1(\Rd)$), and that in $D$ there is no boundary condition linking the regular and the singular part. We observe that $Q(v)=\left\langle v,H_{\alpha}v\right\rangle$, whenever $u\in D(H_{\alpha})$.

Let us introduce the main object of our study, the energy functional associated with the $\delta$-NLSE, which is another quantity conserved by the flow generated by \eqref{tdNLS} (\cite{CFN-21}).

\begin{definition}
Given $\alpha\in\R$ and $p>2$, the \emph{$\delta$-NLS energy} is the functional $E:D\to\R$ defined by
\begin{align}
\label{E}
E(v):= & \f{1}{2}Q(v)-\f{1}{p}\|v\|_p^p\nonumber\\
= & \f{1}{2}\|\na\phi_\la\|^2_{2}+\f{\la}{2}\big(\|\phi_\la\|^2_{2}-\|v\|^2_{2}\big)+\f{\left(\alpha+\theta_\la\right)}{2}|q|^2-\f{1}{p}\|v\|^{p}_{p}.
\end{align}
\end{definition}

\begin{remark} \label{extension}
As a peculiar feature of point interactions in dimensions two and three, the energy space $D$ is strictly larger than $H^1$. Furthermore, if $v$ belongs to $H^1$, i.e. it has no charge,
then its energy reduces to the standard
\emph{NLS energy} defined by
\begin{equation}
\label{Epclass}
E^{0}(v):=\f{1}{2}\|\na v\|^2_{2}-\f{1}{p}\|v\|^{p}_{p},
\end{equation}
so that the $\delta$-{\em NLS} energy is an extension of the {\em NLS} energy.
\end{remark}

We can now give the following fundamental definition.

\begin{definition}
Given $\mu>0$,  a function $u$ belonging to the space
\begin{equation*}
D_{\mu}:=\{v\in D:\|v\|_{2}^{2}=\mu\}.
\end{equation*}
and satisfying
%\in D_{\mu}$ such that
\begin{equation*}
E(u)=\inf_{v\in D_{\mu}}E(v)=:\Eps(\mu),
\end{equation*}
is a {\em ground state at mass $\mu$} for the NLSE with a point defect.
\end{definition}
Thus, a ground state is a minimizer of the energy constrained to a submanifold of constant mass $\mu$. It turns out that the whole set of ground states at mass $\mu$ is orbitally stable for any $\mu>0$: the result is proven by adapting the techniques introduced in \cite{CL-82} and is reported in Appendix \ref{app-stab}.

\medskip
We can now state the main result of the paper. 

\begin{theorem}[$\delta$-NLS ground states]
\label{exchar-gs}
Let $p\in(2,4)$ and $\alpha\in \R$. Then, for every $\mu>0$,
\begin{itemize}
 \item[(i)] there exists a ground state for the $\delta$-NLS at mass $\mu$;
 \item[(ii)] if, fixed $\lambda > 0$, $u=\phi_\la+q\G_\la$ is a  ground state, then:
 \begin{itemize}
  \item[(a)] for any $\lambda>0$ both $\phi_\la$ and $q$ are not identically zero,
  \item[(b)] $u$ is positive, radially symmetric, and decreasing along the radial direction, up to multiplication by a constant phase; moreover, $\phi_\la$ is nonnegative if $\lambda=\omega$, and positive if $\lambda>\omega$, with $\omega=\mu^{-1}(\|u\|_p^p-Q(u))$.
 \end{itemize}
 \end{itemize}
\end{theorem}

\begin{remark}
 One can also see that if $u$ is a ground state for the $\delta$-NLSE, then the associated frequency $\omega=\mu^{-1}(\|u\|_p^p-Q(u))$ is positive. Indeed, by the Lagrange Multiplier Theorem, one has
 \[
  \langle E'(u),v\rangle+\omega\langle u,v\rangle=0,\qquad\forall v\in D,
 \]
 so that, setting $v=u$ and combining with \eqref{ELdelta}, \eqref{eq-Q} and \eqref{E},
 \begin{equation}
 \label{eq-eneq}
  2E(u)-\frac{p-2}{p}\|u\|_p^p=-\omega\|u\|_2^2.
 \end{equation}
 Then, by Remark \ref{extension},
 \begin{equation*}
  E(u)=\Eps(\mu)\leq\Eps^{0}(\mu):=\inf_{v\in H^{1}_{\mu}(\Rd)}E^0(v),
 \end{equation*}
 with $H^{1}_{\mu}(\Rd):=\{v\in H^1(\Rd):\|v\|_2^2=\mu\}$,  which is negative whenever $p\in(2,4)$ (see, e.g., \cite{L-ANIHPC84} or the proof of Proposition \ref{exstanden}). Thus \eqref{eq-eneq} implies that $\omega>0$ .
\end{remark}

We stress that Theorem \ref{exchar-gs} treats power nonlinearities with $p\in(2,4)$ only, namely the  subcritical nonlinearities of the NLSE, since as in the standard case this is necessary to establish  boundedness from below of the constrained energy without prescriptions on the mass $\mu$ (see Proposition \ref{Epabound}). In addition, also regarding existence, positivity, and symmetry, Theorem \ref{exchar-gs} retraces the results on the NLSE. However, point (ii)(a) shows that $\delta$-NLS ground states and NLS ground states cannot coincide as the singular part of the former ones cannot vanish.

As a final remark, we highlight that, while point (i) of Theorem \ref{exchar-gs} is proved by minimization of $E$ on $D_\mu$, point (ii) is proved through  minimization of another functional, called {\em Action}, constrained to a set called {\em Nehari manifold}. More precisely, %the following functional
\begin{definition}
Fixed $\omega\in\R$, the $\delta$-{NLS action} at frequency $\omega$ is the functional $S_\omega:D\to\R$ defined by
\begin{equation}
\label{SwE}
S_\omega(v):=E(v)+\f{\omega}{2}\|v\|_2^2.
\end{equation} 
\end{definition}
We introduce the constraint
\begin{definition}
Fixed $\omega \in \R$, the Nehari manifold at frequency $\omega$ associated to the $\delta$-{\em NLS} is defined by
\begin{equation}
\label{eq-nehari}
N_{\omega}:=\{v\in D\setminus\{0\}:I_{\omega}(v)=0\},
\end{equation}
where $I_\omega:D\to\R$ denotes the quantity
\begin{equation}
\label{eq-Inehari}
I_{\omega}(v):=\langle S_{\omega}'(v),v\rangle=\|\na\phi_\la\|^2_{2}+\la\|\phi_\la\|^2_{2}+(\omega-\la)\|v\|^2_{2}+\left(\alpha+\theta_\la\right)|q|^2-\|v\|^{p}_{p}.
\end{equation}
\end{definition}
As a consequence, the minimizers of the \emph{$\delta$-NLS} action at frequency $\omega$ are all functions $u\in N_{\omega}$ such that
\begin{equation}
\label{eq-infact}
S_{\omega}(u)=d(\omega):=\inf_{v\in N_{\omega}}S_{\omega}(v),
\end{equation}
and, as showed in Appendix \ref{app-gbstates}, they are bound states of the $\delta$-NLS.

\begin{remark}
\label{rem-qomega}
We  use the notation
\begin{equation}
\label{Qomega}
Q_\omega(v):=Q(v)+\omega\|v\|_2^2,
\end{equation} 
so that
\begin{equation}
\label{eq-utile}
 S_\omega(v)=\f{1}{2}Q_\omega(v)-\f{1}{p}\|v\|_p^p\qquad\text{and}\qquad I_\omega(v)=Q_\omega(v)-\|v\|_p^p.
\end{equation}
\end{remark}

The link between minimizers of the action and ground states is provided by the following Lemma, whose proof is an adaptation of what established in \cite{DST-21} and \cite{JL-21} for the NLSE. We report it in Appendix \ref{app-enact}.

\begin{lemma}
\label{chargs}
Let $p>2$, $\alpha\in\R$ and $\mu>0$. If $u$ is a ground state of mass $\mu$, then it is also a minimizer of the action at the frequency $\omega=\mu^{-1}(\|u\|_p^p-Q(u))$.
\end{lemma}

We give the following result for the minimizers of the action functional.

\begin{theorem}[$\delta$-NLS action minimizers]
\label{exchar-actmin}
Let $p>2$ and $\alpha\in \R$. Then, 
\begin{itemize}
\item[(i)] a minimizer of the action \eqref{SwE}  at frequency $\omega$ does exist if and only if $\omega>\omega_0:=-\ell_\alpha$, with $\ell_\alpha$ defined in \eqref{eq-spectrum};

\item[(ii)] if, fixed $\la$, $u=\phi_\la+q\G_\la$ is a minimizer of the action \eqref{SwE} at frequency $\omega>\omega_0$, then:
\begin{itemize}
 \item[(a)] for any $\lambda>0$ both $\phi_\la$ and $q$ are not identically zero,
 \item[(b)] $u$ is positive, radially symmetric, and decreasing along the radial direction, up to multiplication by a constant phase factor; in particular, $\phi_\la$ is nonnegative if $\lambda=\omega$, and positive if $\lambda>\omega$.
\end{itemize}
\end{itemize}
\end{theorem}

First, we note that, in view of Lemma \ref{chargs}, point (ii) of Theorem \ref{exchar-gs} is a straightforward consequence of Theorem \ref{exchar-actmin}. Indeed, if there exists a ground state of mass $\mu$, then by Lemma \ref{chargs} and point (i) of Theorem \ref{exchar-actmin} it is also a minimizer of the action at frequency $\omega>\omega_0$. Hence, the conclusion follows by point (ii) of Theorem \ref{exchar-actmin}. %Moreover, the comparisons with the NLSE and the C-NLSE made after Theorem \ref{exchar-gs} are still valid.
  
Furthermore, we mention that in order to establish (ii)(b) we use an equivalent formulation of the problem of minimization of the action consisting in minimizing $Q_\omega$ on the functions in $D$ with fixed $L^p$ norm. More details on this point are given at the beginning of Section \ref{sec-char} and in Remark \ref{no-rearr-gs}. The technique relies on the minimality of the ground states only and is different from other classical techniques, such as the {moving planes} introduced in \cite{GNN-79}, and from more recent variational methods like \cite{FGI-21,JL-21}, where the Euler-Lagrange equation is used to enhance the regularity of the minimizers.

During the final draft of the present paper we got acquainted that the results of Theorem \ref{exchar-actmin} had been proved independently in the contemporary work \cite{FGI-21}.  In particular, except from the overlap of point (i)(a) of Theorem \ref{exchar-actmin} with \cite[Theorem 1.2]{FGI-21}, the proof of point (ii)(b) of Theorem \ref{exchar-actmin} relies on different techniques, as explained above. Moreover, the goals of the two papers are different: while the present paper is mainly focused on the study of ground states of the energy at fixed mass, \cite{FGI-21} deals with the minimizers of the action under the Nehari's constraint, discussing their orbital stability for asymptotic regimes of the frequency $\omega$. 

%%%%%%%%%%%%%%%%%%%%%%%%%%%%%%%%%%%%%%%%%%%%%%%%%%%%%%%%%%%%%%%%%%%%%%%%%%%%%%%%
%%%%%%%%%%%%%%%%%%%%%%%%%%%%%%%%%%%%%%%%%%%%%%%%%%%%%%%%%%%%%%%%%%%%%%%%%%%%%%%%
%%%%%%%%%%%%%%%%%%%%%%%%%%%%%%%%%%%%%%%%%%%%%%%%%%%%%%%%%%%%%%%%%%%%%%%%%%%%%%%%
%%%%%%%%%%%%%%%%%%%%%%%%%%%%%%%%%%%%%%%%%%%%%%%%%%%%%%%%%%%%%%%%%%%%%%%%%%%%%%%%
%%%%%%%%%%%%%%%%%%%%%%%%%%%%%%%%%%%%%%%%%%%%%%%%%%%%%%%%%%%%%%%%%%%%%%%%%%%%%%%%
%%%%%%%%%%%%%%%%%%%%%%%%%%%%%%%%%%%%%%%%%%%%%%%%%%%%%%%%%%%%%%%%%%%%%%%%%%%%%%%%

\subsection*{Organization of the paper}

\begin{itemize}
 \item %[(a)] 
Section \ref{sec-prel} introduces some preliminary results that are useful throughout the paper; more precisely:
 \begin{itemize}
  \item in Section \ref{subsec-green} we recall some well-known features of the Green's function of $-\Delta+\lambda$,
  \item in Section \ref{subsec-GN} we establish two extensions of the Gagliardo-Nirenberg inequality (Proposition \ref{GNsing}),
  \item in Section \ref{subsec-rearr} we establish a rearrangement inequality for the $L^p$-norms of the sum of nonnegative functions (Proposition \ref{fgtheor});
 \end{itemize}
 \item %[(b)] 
Section \ref{sec-gs} addresses the existence of  ground states (Theorem \ref{exchar-gs}--(i));
 \item %[(c)] 
Section \ref{sec-am} addresses the existence of action minimizers (Theorem \ref{exchar-actmin}--(i));
 \item %[(d)] 
Section \ref{sec-char} establishes the main features both of the $\delta$-NLS ground states and of the action minimizers (Theorem \ref{exchar-gs} --(ii)/Theorem \ref{exchar-actmin}--(ii)).
\end{itemize}

%%%%%%%%%%%%%%%%%%%%%%%%%%%%%%%%%%%%%%%%%%%%%%%%%%%%%%%%%%%%%%%%%%%%%%%%%%%%%%%%
%%%%%%%%%%%%%%%%%%%%%%%%%%%%%%%%%%%%%%%%%%%%%%%%%%%%%%%%%%%%%%%%%%%%%%%%%%%%%%%%
%%%%%%%%%%%%%%%%%%%%%%%%%%%%%%%%%%%%%%%%%%%%%%%%%%%%%%%%%%%%%%%%%%%%%%%%%%%%%%%%
%%%%%%%%%%%%%%%%%%%%%%%%%%%%%%%%%%%%%%%%%%%%%%%%%%%%%%%%%%%%%%%%%%%%%%%%%%%%%%%%
%%%%%%%%%%%%%%%%%%%%%%%%%%%%%%%%%%%%%%%%%%%%%%%%%%%%%%%%%%%%%%%%%%%%%%%%%%%%%%%%
%%%%%%%%%%%%%%%%%%%%%%%%%%%%%%%%%%%%%%%%%%%%%%%%%%%%%%%%%%%%%%%%%%%%%%%%%%%%%%%%

\subsection*{Acknowledgements}

The authors thank Simone Dovetta for helpful discussions and the anonymous referees for suggesting how to improve the paper. In particular, we added Appendix \ref{app-stab} in order to fulfil their advices.

The work was partially supported by the MIUR project “Dipartimenti di Eccellenza 2018-2022” (CUP E11G18000350001) and by the INdAM GNAMPA project 2020 “Modelli differenziali alle derivate parziali per fenomeni di interazione”.

%%%%%%%%%%%%%%%%%%%%%%%%%%%%%%%%%%%%%%%%%%%%%%%%%%%%%%%%%%%%%%%%%%%%%%%%%%%%%%%%
%%%%%%%%%%%%%%%%%%%%%%%%%%%%%%%%%%%%%%%%%%%%%%%%%%%%%%%%%%%%%%%%%%%%%%%%%%%%%%%%
%%%%%%%%%%%%%%%%%%%%%%%%%%%%%%%%%%%%%%%%%%%%%%%%%%%%%%%%%%%%%%%%%%%%%%%%%%%%%%%%
%%%%%%%%%%%%%%%%%%%%%%%%%%%%%%%%%%%%%%%%%%%%%%%%%%%%%%%%%%%%%%%%%%%%%%%%%%%%%%%%
%%%%%%%%%%%%%%%%%%%%%%%%%%%%%%%%%%%%%%%%%%%%%%%%%%%%%%%%%%%%%%%%%%%%%%%%%%%%%%%%
%%%%%%%%%%%%%%%%%%%%%%%%%%%%%%%%%%%%%%%%%%%%%%%%%%%%%%%%%%%%%%%%%%%%%%%%%%%%%%%%

\subsection*{Data availability statement}

Data sharing not applicable to this article as no datasets were generated or analysed during the current study.

%%%%%%%%%%%%%%%%%%%%%%%%%%%%%%%%%%%%%%%%%%%%%%%%%%%%%%%%%%%%%%%%%%%%%%%%%%%%%%%%
%%%%%%%%%%%%%%%%%%%%%%%%%%%%%%%%%%%%%%%%%%%%%%%%%%%%%%%%%%%%%%%%%%%%%%%%%%%%%%%%
%%%%%%%%%%%%%%%%%%%%%%%%%%%%%%%%%%%%%%%%%%%%%%%%%%%%%%%%%%%%%%%%%%%%%%%%%%%%%%%%
%%%%%%%%%%%%%%%%%%%%%%%%%%%%%%%%%%%%%%%%%%%%%%%%%%%%%%%%%%%%%%%%%%%%%%%%%%%%%%%%
%%%%%%%%%%%%%%%%%%%%%%%%%%%%%%%%%%%%%%%%%%%%%%%%%%%%%%%%%%%%%%%%%%%%%%%%%%%%%%%%
%%%%%%%%%%%%%%%%%%%%%%%%%%%%%%%%%%%%%%%%%%%%%%%%%%%%%%%%%%%%%%%%%%%%%%%%%%%%%%%%
%%%%%%%%%%%%%%%%%%%%%%%%%%%%%%%%%%%%%%%%%%%%%%%%%%%%%%%%%%%%%%%%%%%%%%%%%%%%%%%%
%%%%%%%%%%%%%%%%%%%%%%%%%%%%%%%%%%%%%%%%%%%%%%%%%%%%%%%%%%%%%%%%%%%%%%%%%%%%%%%%
%%%%%%%%%%%%%%%%%%%%%%%%%%%%%%%%%%%%%%%%%%%%%%%%%%%%%%%%%%%%%%%%%%%%%%%%%%%%%%%%
%%%%%%%%%%%%%%%%%%%%%%%%%%%%%%%%%%%%%%%%%%%%%%%%%%%%%%%%%%%%%%%%%%%%%%%%%%%%%%%%
%%%%%%%%%%%%%%%%%%%%%%%%%%%%%%%%%%%%%%%%%%%%%%%%%%%%%%%%%%%%%%%%%%%%%%%%%%%%%%%%
%%%%%%%%%%%%%%%%%%%%%%%%%%%%%%%%%%%%%%%%%%%%%%%%%%%%%%%%%%%%%%%%%%%%%%%%%%%%%%%%

%%%%%%%%%%%%%%%%%%%%%%%%%%%%%%%%%%
%%%%%%%%%% Preliminari %%%%%%%%%%%
%%%%%%%%%%%%%%%%%%%%%%%%%%%%%%%%%%

\section{Preliminary results}
\label{sec-prel}

In this section we collect some preliminary results, that will be exploited in the proofs of Theorem \ref{exchar-gs} and Theorem \ref{exchar-actmin}.

%%%%%%%%%%%%%%%%%%%%%%%%%%%%%%%%%%%%%%%%%%%%%%%%%%%%%%%%%%%%%%%%%%%%%%%%%%%%%%%%
%%%%%%%%%%%%%%%%%%%%%%%%%%%%%%%%%%%%%%%%%%%%%%%%%%%%%%%%%%%%%%%%%%%%%%%%%%%%%%%%
%%%%%%%%%%%%%%%%%%%%%%%%%%%%%%%%%%%%%%%%%%%%%%%%%%%%%%%%%%%%%%%%%%%%%%%%%%%%%%%%
%%%%%%%%%%%%%%%%%%%%%%%%%%%%%%%%%%%%%%%%%%%%%%%%%%%%%%%%%%%%%%%%%%%%%%%%%%%%%%%%
%%%%%%%%%%%%%%%%%%%%%%%%%%%%%%%%%%%%%%%%%%%%%%%%%%%%%%%%%%%%%%%%%%%%%%%%%%%%%%%%
%%%%%%%%%%%%%%%%%%%%%%%%%%%%%%%%%%%%%%%%%%%%%%%%%%%%%%%%%%%%%%%%%%%%%%%%%%%%%%%%

\subsection{Further properties of the Green's function}
\label{subsec-green}

First, we recall that $\G_\la\in L^2(\Rd)\setminus H^1(\Rd)$, is positive, radially symmetric,  decreasing along the radial direction, has exponential decay at infinity, and is smooth up to the origin, where it satisfies
\[
 \G_\la(\x)=-\frac{1}{2\pi}\log\left(\frac{\sqrt{\lambda}|\x|}{2}\right)+\frac{\gamma}{2\pi}+o(1),\qquad\text{as}\quad|\x|\to \z
\]
(see \cite[Sec. 9.6]{AS-65} and \cite[Eq. (1.5)]{CCF-17}). Thus $\G_\lambda$ belongs to $L^p(\R^2)$, $2 \leq p < \infty$. Moreover,
\begin{equation}
\label{eq-GLP}
 \|\G_{\la}\|_{2}^{2}=\f{1}{4\pi\la}\qquad\text{and}\qquad \|\G_{\la}\|_{p}^{p}= \f{\|\G_{1}\|_p^p}{\la}.
\end{equation}
On the other hand, using \eqref{eq-green}, one can prove that $\G_{\la}-\G_{\nu}\in H^2(\Rd)$. Also, by direct computation,
\begin{gather}
\label{eq-Glanu}\|\G_{\la}-\G_{\nu}\|_{2}^{2}=\f{1}{4\pi}\left(\f{1}{\la}+\f{1}{\nu}+\f{2\log(\nu/\la)}{\la-\nu}\right),\\
\label{eq-DGlanu}\|\na(\G_{\la}-\G_{\nu})\|_{2}^{2}=\f{1}{4\pi}\left(\f{(\la+\nu)\log(\nu/\la)}{\la-\nu}-2\right).
\end{gather}
Finally, we note that if $\nu<\lambda$, then
\begin{equation}
\label{Gla<Gnu}
\G_\la(\x)=\f{K_0(\sqrt{\la}\x)}{2\pi}=\G_\nu\left(\sqrt{\f{\la}{\nu}}\x\right)<\G_\nu(\x),\qquad\forall\x\in\Rd\setminus\{\z\}.
\end{equation}

%%%%%%%%%%%%%%%%%%%%%%%%%%%%%%%%%%%%%%%%%%%%%%%%%%%%%%%%%%%%%%%%%%%%%%%%%%%%%%%%
%%%%%%%%%%%%%%%%%%%%%%%%%%%%%%%%%%%%%%%%%%%%%%%%%%%%%%%%%%%%%%%%%%%%%%%%%%%%%%%%
%%%%%%%%%%%%%%%%%%%%%%%%%%%%%%%%%%%%%%%%%%%%%%%%%%%%%%%%%%%%%%%%%%%%%%%%%%%%%%%%
%%%%%%%%%%%%%%%%%%%%%%%%%%%%%%%%%%%%%%%%%%%%%%%%%%%%%%%%%%%%%%%%%%%%%%%%%%%%%%%%
%%%%%%%%%%%%%%%%%%%%%%%%%%%%%%%%%%%%%%%%%%%%%%%%%%%%%%%%%%%%%%%%%%%%%%%%%%%%%%%%
%%%%%%%%%%%%%%%%%%%%%%%%%%%%%%%%%%%%%%%%%%%%%%%%%%%%%%%%%%%%%%%%%%%%%%%%%%%%%%%%

\subsection{Extensions of the Gagliardo-Nirenberg inequality}
\label{subsec-GN}

We need a generalization of Gagliardo-Nirenberg inequality to the energy space $D$. 

Let us recall the standard two-dimensional Gagliardo-Nirenberg inequality (\cite[Theorem 1.3.7]{C-CL03}): there exists $C_{p}>0$ such that
\begin{equation}
\label{GNclass}
\|v\|_{p}^{p}\le C_{p}\|\na v\|_{2}^{p-2}\|v\|_{2}^{2}, \qquad\forall\, v \in H^{1}(\Rd).
\end{equation}

First, the set of functions in the energy space with $q \neq 0$ can be written as
\begin{equation}
\label{Sing}
D\setminus H^1(\Rd) =\big\{u\in L^{2}(\Rd):\exists q\in \C\setminus\{0\}\:\text{ s.t. }\:u-q\G_{\f{|q|^{2}}{\|u\|_{2}^{2}}}=:\phi\in H^{1}(\Rd)\big\}.
\end{equation}
In other words, functions in $D\setminus H^1(\Rd)$ admit the decomposition with $\la=\f{|q|^{2}}{\|u\|_{2}^{2}}$, where the right-hand side is well defined as the next remark shows.

\begin{remark}[$H_\al$ and $q$ do not depend on $\la>0$]
\label{qindla}
Let us consider $v\in D(H_\al)$. By definition, there exist $q\in\C$ and $\lambda>0$ such that $v=\phi_\la+q\G_\la$, $\phi_\la\in H^2(\Rd)$ and $\phi_{\la}(0)=(\alpha+\theta_\la)q$. 
Notice that
$$ q : = - 2 \pi \lim_{|\x| \to 0} \frac { v (\x)} {\log |\x|}, $$
therefore $q$ does not depend on $\lambda$. Moreover, choosing $0 < \nu \neq \lambda$, it is possible to decompose the same function as $v = \phi_\nu + q G_\nu$. Since, as mentioned in Section \ref{subsec-green}, $\G_\la-\G_\nu\in H^2(\R^2)$, one gets  that $\phi_\nu:=\phi_\la+q(\G_\la-\G_\nu)$ belongs to $H^2(\R^2)$. Moreover, \eqref{eq-green} also implies
\begin{equation}
 \label{eq-propG2}
 (\G_\la-\G_\nu)(\z)=(4\pi)^{-1}\log(\nu/\la),
\end{equation}
so that $\theta_\la+(\G_\la-\G_\nu)(\z)=\theta_\nu$, whence $\phi_{\nu}(0)=\left(\alpha+\theta_\nu\right)q$. Finally,  by \eqref{eq-green}
\begin{equation}
 \label{eq-propG3}
 -\Delta(\G_\la-\G_\nu)=\nu\G_\nu-\la\G_\la,
\end{equation}
so $-\lap\phi_{\la}-q\la\G_\la=-\lap\phi_{\nu}-q\nu\G_\nu$ and thus the decompositions with $\nu$ and $\la$ are equivalent.\\
\end{remark}

We can now state the following
 \begin{proposition}[Extended Gagliardo-Nirenberg inequalities]
 \label{GNsing}
 For every $p>2$, there exists $K_{p}>0$ such that
 \begin{equation}
 \label{GND}
 \|v\|_p^p\le K_{p}\left(\|\na \phi_\la\|_{2}^{p-2}\|\phi_\la\|_{2}^{2}+\f{|q|^{p}}{\la}\right), \qquad\forall v=\phi_\la+q\G_{\la} \in D,\quad\forall\la>0.
 \end{equation}
 Moreover, there exists $M_{p}>0$
\begin{equation}
\label{GNDs}
\|v\|_{p}^{p}\le M_{p}\left(\|\na \phi\|_{2}^{p-2}+|q|^{p-2}\right)\|v\|_{2}^{2}, \qquad \forall v=\phi+q\G_{\f{|q|^{2}}{\|v\|_{2}^{2}}} \in D\setminus H^1(\Rd).
\end{equation}
\end{proposition}
\begin{proof}
If we fix $v=\phi_\la+q\G_\la\in D$, for some $\la>0$, then \eqref{GNclass} and \eqref{eq-GLP} yield
\begin{equation*}
\|v\|_{p}^{p}=\|\phi_{\la}+q\G_{\la}\|_{p}^{p}\le 2^{p-1}\left(\|\phi_{\la}\|_{p}^{p}+|q|^{p}\|\G_{\la}\|_{p}^{p}\right)\le K_{p}\left(\|\na \phi_{\la}\|_{2}^{p-2}\|\phi_{\la}\|_{2}^{2}+\f{|q|^{p}}{\la}\right),
\end{equation*}
that is \eqref{GND}.

If we suppose, in addition, that $q\neq 0$ and set $\la=\la_q:=\f{|q|^{2}}{\|v\|_{2}^{2}}$, then by \eqref{eq-GLP}, \eqref{GND} and the triangle inequality there results
\begin{equation*}
\|v\|_{p}^{p}\le M_{p}\left(\|\na \phi\|_{2}^{p-2}\|v\|_{2}^{2}+\|\na \phi\|_{2}^{p-2}\f{|q|^{2}}{\la_q}+\f{|q|^{p}}{\la_q}\right)\le M_{p}\left(\|\na \phi\|_{2}^{p-2}+|q|^{p-2}\right)\|v\|_{2}^{2}
\end{equation*}
possibly redefining $M_p$, which concludes the proof.
\end{proof}

\begin{remark}
Note that, whenever $q=0$, i.e. $v\in H^1(\Rd)$, \eqref{GND} reduces to \eqref{GNclass}.
\end{remark}

%%%%%%%%%%%%%%%%%%%%%%%%%%%%%%%%%%%%%%%%%%%%%%%%%%%%%%%%%%%%%%%%%%%%%%%%%%%%%%%%
%%%%%%%%%%%%%%%%%%%%%%%%%%%%%%%%%%%%%%%%%%%%%%%%%%%%%%%%%%%%%%%%%%%%%%%%%%%%%%%%
%%%%%%%%%%%%%%%%%%%%%%%%%%%%%%%%%%%%%%%%%%%%%%%%%%%%%%%%%%%%%%%%%%%%%%%%%%%%%%%%
%%%%%%%%%%%%%%%%%%%%%%%%%%%%%%%%%%%%%%%%%%%%%%%%%%%%%%%%%%%%%%%%%%%%%%%%%%%%%%%%
%%%%%%%%%%%%%%%%%%%%%%%%%%%%%%%%%%%%%%%%%%%%%%%%%%%%%%%%%%%%%%%%%%%%%%%%%%%%%%%%
%%%%%%%%%%%%%%%%%%%%%%%%%%%%%%%%%%%%%%%%%%%%%%%%%%%%%%%%%%%%%%%%%%%%%%%%%%%%%%%%

\subsection{A rearrangement inequality in $L^p$-spaces}
\label{subsec-rearr}

Let us start by recalling the definition of \emph{radially symmetric nonincreasing rearrangement} of a function in $\Rd$ and its main features (see e.g. \cite[Chapter 3]{LL-GSM01}). All the definitions and the results in Section \ref{subsec-rearr} are valid in every $\R^d$, with $d\geq2$.

 First, given a measurable $A\subset \Rd$ with finite Lebesgue measure, we denote by $A^*$ the open ball centred at zero with Lebesgue measure equal to $|A|$, that is
\begin{equation*}
A^*:=\{x\in \Rd:\pi|\x|^2<|A|\}.
\end{equation*}
%with $\omega_2$ the Lebesgue measure of the unit ball of $\R^2$. 
Now, let $f:\Rd\to\R$ be a nonnegative measurable function \emph{vanishing at infinity}, i.e. $|\{f>t\}|:=|\{\x\in\Rd:f(\x)>t\}|<+\infty$, for every $t>0$. We call the \emph{radially symmetric nonincreasing rearrangement} of $f$ the function $f^*:\Rd\to\R$ defined by% the function
\begin{equation}
\label{eq-rearr}
f^*(\x)=\int_0^{\infty} \mathds{1}_{\{f>t\}^*}(\x)\dt,
\end{equation}
with $\mathds{1}_{\{f>t\}^*}$ the characteristic function of $\{f>t\}^*$. Definition \eqref{eq-rearr} clearly implies 
\begin{equation}
\label{eq-equim}
 |\{f>t\}|=|\{f^*>t\}|\qquad\text{and}\qquad\{f>t\}^*=\{f^*>t\}
\end{equation}
and 
\begin{equation}
 \label{eq-funset}
 \mathds{1}_A^*\equiv\mathds{1}_{A^*},\qquad\text{for every measurable}\quad A\subset\Rd,\quad|A|<+\infty.
\end{equation}
One can also check that
\begin{equation}
\label{eq-phirearr}
 (\Phi\circ f)^*\equiv\Phi\circ f^*,\qquad\text{for every nondecreasing}\quad\Phi:\R^+\to \R^+
\end{equation}
and that \eqref{eq-equim} yields
\begin{equation}
\label{equimeas}
\|f^*\|_p=\|f\|_p,\qquad\forall f\in L^p(\Rd),\: f\geq0,\quad \forall\, p\geq1.
\end{equation}

Another well known property of rearrangements is the \emph{Hardy-Littlewood inequality}, which states that, given two nonnegative measurable functions $f,g:\Rd\to\R$ vanishing at infinity, there results
\begin{equation}
\label{HL-ineq}
\int_{\Rd}f(\x)g(\x)\dx\le \int_{\Rd}f^*(\x)g^*(\x)\dx
\end{equation} 
and, if $f$ is radially symmetric and decreasing, then the equality holds in \eqref{HL-ineq} if and only if $g=g^*$ a.e. on $\Rd$.

We also need to compare  $\|f+g\|_p$ and $\|f^*+g^*\|_p$ and a related result is stated in the next Proposition. The first statement is actually a special case of a well known result proved in \cite[Theorem 2.2]{AL-89}. The second statement is a refinement of that result and, as far as we know, no proof of it is present in the literature. Our proof adapts the arguments used in \cite[Theorems 3.4 and 3.5]{LL-GSM01}.

\begin{proposition}[Rearrangement inequality]
\label{fgtheor}
For every pair of nonnegative functions $f$, $g\in L^p(\Rd)$, with $p>1$, there results 
\begin{equation}
\label{ineqf+g}
\int_{\Rd} |f+g|^p \dx\le \int_{\Rd} |f^*+g^*|^p\dx.
\end{equation} 
Moreover, if $f$ is radially symmetric and strictly decreasing, then the equality in \eqref{ineqf+g} implies that $g=g^*$ a.e. on $\Rd$.
\end{proposition}

\begin{proof}
First, we introduce the function
\begin{equation*}
J_+(t):=
\begin{cases}
J(t)\,&\text{if}\quad t\ge 0,\\
0 &\text{if}\quad t<0,
\end{cases}
\qquad\text{with}\qquad J(t):=|t|^p.
\end{equation*}
It is straightforward that $J_+$ is of class $C^1$, with $J'_+$ nonnegative and nondecreasing in $\R$ and, in particular, positive and increasing in $\R^+$. Therefore, 
\begin{equation*}
|f(\x)+g(\x)|^p=J_+(f(\x)+g(\x))=\int_{-f(\x)}^{g(\x)}J'_+(f(\x)+s)\ds=\int_{-\infty}^{+\infty} J'_+(f(\x)+s)\mathds{1}_{\{g>s\}}(\x)\ds,
\end{equation*}
whence, integrating over $\Rd$ and using Tonelli's theorem, we get
\begin{multline}
\label{eq-lpnorm}
\int_{\Rd}|f(\x)+g(\x)|^p\dx\\
=\int_{\Rd}J_+(f(\x)+g(\x))\dx=\int_{-\infty}^{+\infty}\left(\int_{\Rd} J'_+(f(\x)+s)\mathds{1}_{\{g>s\}}(\x)\dx\right)\ds\\
=\underbrace{\int_0^{+\infty}\int_{\R^2}J'_+(f(\x)-s)\dx\ds}_{=:I_1}+\underbrace{\int_0^{+\infty}\int_{\R^2}J'_+(f(\x)+s)\mathds{1}_{\{g>s\}}(\x)\dx\ds}_{=:I_2}
\end{multline}
where we used the fact that $\mathds{1}_{\{g>-s\}}\equiv1$ for every $s>0$. Now, combining \eqref{eq-equim} and \eqref{eq-phirearr} with $\Phi(\cdot)=J_+'(\cdot-s)$, one sees that
\begin{equation}
\label{eq-I1}
 I_1=\int_0^{+\infty}\int_{\R^2}J'_+(f^*(\x)-s)\dx\ds.
\end{equation}
On the other hand, combining \eqref{eq-equim}, \eqref{eq-funset}, \eqref{eq-phirearr} with $\Phi(\cdot)=J_+'(\cdot+s)-J_+'(s)$ and \eqref{HL-ineq}, one sees that
\begin{align}
\label{eq-hardy}
 \int_{\R^2}J'_+(f(\x)+s)\mathds{1}_{\{g>s\}}(\x)\dx & \nonumber\\
 & \hspace{-3cm} =\int_{\R^2}\big(J'_+(f(\x)+s)-J_+'(s)\big)\mathds{1}_{\{g>s\}}(\x)\dx+J_+'(s)|\{g>s\}|\nonumber\\
 & \hspace{-3cm} \leq \int_{\R^2}\big(J'_+(f(\x)+s)-J_+'(s)\big)^*\mathds{1}_{\{g>s\}}^*(\x)\dx+J_+'(s)|\{g^*>s\}|\nonumber\\
 & \hspace{-3cm} = \int_{\R^2}\big(J'_+(f^*(\x)+s)-J_+'(s)\big)\mathds{1}_{\{g^*>s\}}(\x)\dx+J_+'(s)|\{g^*>s\}|\nonumber\\
 & \hspace{-3cm} =\int_{\Rd}J_+'(f^*+(\x)+s)\mathds{1}_{\{g^*>s\}}(\x)\dx,
\end{align}
so that
\begin{equation}
\label{eq-I2}
 I_2\leq\int_0^{+\infty}\int_{\Rd}J_+'(f^*+(\x)+s)\mathds{1}_{\{g^*>s\}}(\x)\dx\ds.
\end{equation}
Hence, in view of \eqref{eq-lpnorm}, \eqref{eq-I1} and \eqref{eq-I2} one easily finds that \eqref{ineqf+g} is satisfied.

It is left to prove that, if $f$ is radially symmetric decreasing and the equality is fulfilled in \eqref{ineqf+g}, then $g=g^*$ a.e. on $\Rd$. To this aim, fix $f$ radially symmetric and decreasing and assume that the equality in \eqref{ineqf+g} holds. Then, one can check that $f=f^*$ a.e. on $\Rd$ and that, by \eqref{eq-hardy},
\begin{equation}
\label{intRdequal}
\int_{\Rd} J'_+(f(\x)+s)\mathds{1}_{\{g>s\}}(\x)\dx=\int_{\Rd} J'_+(f(\x)+s)\mathds{1}_{\{g^*> s\}}(\x)\dx,\qquad\text{for a.e.}\quad s\geq0.
\end{equation}
Since $J'_+$ is increasing on $\R^+$ and  and $f$ is radially symmetric decreasing, $J'_+(f(\cdot)+s)$ is radially symmetric decreasing too. Hence there exists a continuous bijection $r:\R^+\to\R^+$ such that $\{\x\in\Rd:J'_+(f(\x)+s)-J_+'(s)>t\}=B_{r(t)}(\z)$, namely the centered ball of radius $r(t)$. In addition, by dominated converge, the function
\begin{equation*}
F_C(t):=\int_{\Rd}\mathds{1}_{B_{r(t)}(\z)}(\x)\mathds{1}_C(\x)\dx=|B_{r(t)}(\z)\cap C|
\end{equation*}
is continuous on $\R^+$ for any measurable $C\subset\Rd$ fixed. 
 
Now, fix $s>0$ such that \eqref{intRdequal} holds and set $C=\{\x\in\Rd:g(\x)>s\}$. Arguing as before, one can find that $F_C(t)\le F_{C^*}(t)$. From \eqref{eq-hardy} and \eqref{intRdequal}, using the layer-cake representation, one obtains that $\int_0^{\infty} F_C(t)\dt=\int_0^{\infty} F_{C^*}(t)\dt$ and, hence, $F_C(t)=F_C^*(t)$ for every $t>0$. As $C^*$ is a centered ball too, this implies that for every $r>0$ either $C,C^*\subset B_{r}(\z)$ or $C,C^*\supset B_{r}(\z)$ up to sets of zero Lebesgue measure, so that $C=C^*$. Finally, as this is valid for every $s>0$, using again \eqref{eq-equim} and the layer-cake representation, there results that $g=g^*$ a.e. on $\Rd$.
\end{proof}

\begin{remark}
 Up to some minor modifications, in order to prove the first part of Proposition \ref{fgtheor} it suffices the simple convexity of $J$, the strict convexity being necessary for the sole second part. Hence, \eqref{ineqf+g} holds also for $p=1$.
\end{remark}

Before concluding the section, we also mention another well known result on rearrangements that will be used in the sequel: if $f\in H^1(\Rd)$, then $f^*\in H^1(\Rd)$ and in particular
\begin{equation}
\label{PS}
\|\na f^*\|_2\le \|\na f\|_2.
\end{equation}
Equation \eqref{PS} is usually called the \emph{P\'olya-Szeg\H{o} inequality}.

%%%%%%%%%%%%%%%%%%%%%%%%%%%%%%%%%%%%%%%%%%%%%%%%%%%%%%%%%%%%%%%%%%%%%%%%%%%%%%%%
%%%%%%%%%%%%%%%%%%%%%%%%%%%%%%%%%%%%%%%%%%%%%%%%%%%%%%%%%%%%%%%%%%%%%%%%%%%%%%%%
%%%%%%%%%%%%%%%%%%%%%%%%%%%%%%%%%%%%%%%%%%%%%%%%%%%%%%%%%%%%%%%%%%%%%%%%%%%%%%%%
%%%%%%%%%%%%%%%%%%%%%%%%%%%%%%%%%%%%%%%%%%%%%%%%%%%%%%%%%%%%%%%%%%%%%%%%%%%%%%%%
%%%%%%%%%%%%%%%%%%%%%%%%%%%%%%%%%%%%%%%%%%%%%%%%%%%%%%%%%%%%%%%%%%%%%%%%%%%%%%%%
%%%%%%%%%%%%%%%%%%%%%%%%%%%%%%%%%%%%%%%%%%%%%%%%%%%%%%%%%%%%%%%%%%%%%%%%%%%%%%%%
%%%%%%%%%%%%%%%%%%%%%%%%%%%%%%%%%%%%%%%%%%%%%%%%%%%%%%%%%%%%%%%%%%%%%%%%%%%%%%%%
%%%%%%%%%%%%%%%%%%%%%%%%%%%%%%%%%%%%%%%%%%%%%%%%%%%%%%%%%%%%%%%%%%%%%%%%%%%%%%%%
%%%%%%%%%%%%%%%%%%%%%%%%%%%%%%%%%%%%%%%%%%%%%%%%%%%%%%%%%%%%%%%%%%%%%%%%%%%%%%%%
%%%%%%%%%%%%%%%%%%%%%%%%%%%%%%%%%%%%%%%%%%%%%%%%%%%%%%%%%%%%%%%%%%%%%%%%%%%%%%%%
%%%%%%%%%%%%%%%%%%%%%%%%%%%%%%%%%%%%%%%%%%%%%%%%%%%%%%%%%%%%%%%%%%%%%%%%%%%%%%%%
%%%%%%%%%%%%%%%%%%%%%%%%%%%%%%%%%%%%%%%%%%%%%%%%%%%%%%%%%%%%%%%%%%%%%%%%%%%%%%%%

%%%%%%%%%%%%%%%%%%%%%%%%%%%%%%%%%%
%%%%%%%%% Ground states %%%%%%%%%%
%%%%%%%%%%%%%%%%%%%%%%%%%%%%%%%%%%

\section{Ground states existence: proof of Theorem \ref{exchar-gs}--(i)}
\label{sec-gs}

In this section, we prove point $(i)$ of Theorem \ref{exchar-gs}, that is the existence of $\delta$-NLS ground states of mass $\mu$ for every $\mu>0$.

To this aim, the first step is to establish boundedness from below of $E_{|_{D_\mu}}$ in the $L^2(\Rd)$-subcritical case. Preliminarily, it is convenient to write the functional $E$  as
\begin{multline}
\label{Epala}
E(u)=\\[.0cm]
\left\{
\begin{array}{ll}
\displaystyle \f{1}{2}\|\na\phi\|^2_{2}+\f{|q|^{2}\|\phi\|^2_{2}}{2\|u\|_{2}^{2}}+\bigg(\alpha-1+\f{\log\left(\f{|q|}{2\|u\|_{2}}\right)+\ga}{2\pi}\bigg)\f{|q|^2}{2}-\f{\|u\|^{p}_{p}}{p}, &\text{if } u\in  D\setminus H^1(\Rd),\\[.6cm]
\displaystyle \f{1}{2}\|\na u\|^2_{2}-\f{1}{p}\|u\|^{p}_{p}, &\text{if } u \in H^{1}(\Rd),
\end{array}
\right.
\end{multline}
where we use the decomposition  $u=\phi+q\G_{\f{|q|^2}{\|u\|_2^2}}$ introduced in \eqref{Sing}, for every $u\in D\setminus H^1(\Rd)$.

\begin{proposition}
\label{Epabound}
Let $p\in(2,4)$ and $\alpha\in\R$. Then $\Eps(\mu)>-\infty$, for every $\mu>0$.
\end{proposition}

\begin{proof}
Let $u\in D_\mu$. Assume, first, that $u\in H_\mu^{1}(\Rd)$. Therefore, \eqref{GNclass} entails
\begin{equation*}
E(u)\ge \f{1}{2}\|\na u\|^2_{2}-\f{C_{p}}{p}\|\na u\|^{p-2}_{2}\mu,
\end{equation*}
so that $E$ is bounded from below on $H_\mu^{1}(\Rd)$ since $2<p<4$.

Then, assume that $u\in D_\mu\setminus H_\mu^1(\Rd)$. By \eqref{GNDs}
\begin{multline}
\label{Eun-prel}
E(u)\ge \left (\f{1}{2}\|\na\phi\|^2_{2}-\f{M_{p}}{p}\|\na\phi\|^{p-2}_{2}\mu\right)+\f{|q|^{2}\|\phi\|^2_{2}}{2\mu}\\
+\left[\left(\alpha-1+\f{\log\left(\f{|q|}{2\sqrt{\mu}}\right)+\ga}{2\pi}\right)\f{|q|^2}{2}-\f{M_{p}}{p}|q|^{p-2}\mu\right],
\end{multline}
and here again $E$ is bounded from below in $D_\mu\setminus H_\mu^1(\Rd)$ since $2<p<4$ (note that the $\log(|q|)|q|^2$ term balances the negatively diverging $|q|^2$ term). Summing up, $E$ is lower bounded on the whole $D_\mu$.
\end{proof}

Further than boundedness from below, it is also useful to establish a comparison between the $\delta$-NLS energy infimum and the NLS energy infimum.

\begin{proposition}
\label{Eps<Eps0}
Let $p\in(2,4)$ and $\alpha\in\R$. Then, 
\begin{equation}
\label{eq-levcomp}
\Eps(\mu)<\Eps^{0}(\mu)<0,\quad\forall \mu>0.
\end{equation} 
\end{proposition}

In order to prove this, we preliminarily recall without proof a well known result about NLS ground states (see \cite[Theorem II.5]{L-ANIHPC84} for a proof of the existence part, while the proof of the properties satisfied by the ground states is a consequence of \cite[Theorem 2]{GNN-81}).

\begin{proposition}
\label{exstanden}
Let $p\in(2,4)$ and $\mu>0$. Then, there exists a NLS ground state of mass $\mu$, i.e. $u\in H^1_{\mu}(\Rd)$ such that $E^0(u)=\Eps^0(\mu)$. Moreover, such minimizer is unique, positive and radially symmetric decreasing, up to multiplication by a constant phase and translation.
\end{proposition}
The positive minimizer of the two-dimensional standard NLS functional at mass $\mu$ is usually called two-dimensional {\em soliton} 
and in the following it will be denoted by $S_\mu$.

\begin{proof}[Proof of Proposition \ref{Eps<Eps0}]
Fix $\mu>0$ and let $S_\mu$ be the unique NLS ground state of mass $\mu$ mentioned in Proposition \ref{exstanden}. First, note that, as $S_\mu$ is positive, it cannot be a $\delta$-NLS ground state of mass $\mu$. Indeed, if $S_\mu$ is a $\delta$-NLS ground state, then $S_\mu$ has to satisfy \eqref{eq-regbound} and, in particular, $\phi_\la(\z)=(\alpha+\theta_\la)q$. However, as mentioned in Section \ref{subsec-GN}, $S_\mu\in H^1(\Rd)$ implies $q=0$, so that $S_\mu\equiv\phi_\la$ and $\phi_\la(\z)=0$. Hence, $S_\mu(\z)=0$, which contradicts its positivity. Summing up, $S_\mu$ is not a $\delta$-NLS ground state at mass $\mu$ and, thus, there exists $v\in D_\mu$ such that $E(v)<E(S_\mu)=\Eps^{0}(\mu)$, which proves the left inequality in \eqref{eq-levcomp}.

Concerning the right inequality, fix again $\mu>0$, and consider $v\in H^{1}_{\mu}(\Rd)$. Now, using the mass-preserving transformation
 \begin{equation*} 
v_{\sigma}(x)=\sigma v(\sigma x),
\end{equation*}
there results
\begin{equation*}
E^0(v_{\sigma})=\f{\sigma^{2}}{2}\|\na v\|_{2}^{2}-\f{\sigma^{p-2}}{p}\|v\|_{p}^{p}.
\end{equation*}
However, as $p\in(2,4)$, this immediately entails that $\Eps^{0}(\mu)\le E^{0}(v_{\sigma})<0$, for every $\sigma\ll1$, which completes the proof.
\end{proof}

The second step of the proof of point (i) in Theorem \ref{exchar-gs} consists in a characterization of the \emph{$\delta$-NLS energy minimizing sequences of mass $\mu$}, i.e. sequences
\[
 (u_n)_n\subset D_\mu\qquad\text{such that}\qquad E(u_n)\to\Eps(\mu),\quad\text{as}\:n\to+\infty.
\]
This is provided by the next two lemmas.

\begin{lemma} \label{q>C}
\label{minimseq1}
Let $p\in(2,4)$, $\alpha\in\R$ and $\mu>0$. If $u_n = \phi_{\lambda,n} + q_n \G_\lambda$ is a minimizing sequence  
for the $\delta$-NLS energy, then there exists $\bar{n}\in \N$ and a constant $C > 0$, such that $|q_n| > C$ for every $n\ge \bar{n}$.
\end{lemma}

\begin{proof}
We proceed by contradiction. Suppose that there exists a subsequence of $q_{n}$, that we do not rename, such that $q_n \to 0$. Then, $\| \phi_{\lambda,n} \|_2^2$ is bounded since it converges to $\mu$.
Moreover,
applying Gagliardo-Nirenberg inequality \eqref{GND} to definition \eqref{E} one obtains
\begin{eqnarray*}
E (u_n)&\geq &\f 1 2 \| \nabla \phi_{\lambda,n} \|_2^2 + \f \lambda 2 ( \| \phi_{\lambda,n} \|_2^2 - \mu ) 
+ \f{( \alpha + \theta_\lambda)} 2 |q_n |^2 - \f {C_p} p \left(\|\nabla \phi_{\lambda,n} \|_2^{p-2} \| \phi_{\lambda,n}\|_2^2+\f{|q|^{p}}{\la}\right)\\
&=&\f 1 2 \| \nabla \phi_{\lambda,n} \|_2^2 + \f \lambda 2 ( \| \phi_{\lambda,n} \|_2^2 - \mu ) 
+ \f{( \alpha + \theta_\lambda)} 2 |q_n |^2 +o(1)
\end{eqnarray*}
that guarantees the boundedness of $\| \nabla \phi_{\lambda,n} \|_{2}$, since $E(u_n)$ is bounded from above and $p<4$.

We introduce the sequence $ \xi_n = \frac {\sqrt \mu}{\| \phi_{\lambda,n} \|_2}\phi_{\lambda,n}$, such that $\| \xi_n \|_2^2 = \mu$ and 
$ \| \nabla \xi_n \|_2^2 = \frac \mu  {\| \phi_{\lambda,n} \|^2_2} \| \nabla \phi_{\lambda,n} \|^2_2 $ is bounded.
Then, using that and the fact that $q_n \to 0$ and $\phi_{\lambda,n} - u_n \to 0$ strongly in every space $L^p (\Rd)$ with $2 \leq p < \infty$, one obtains
\begin{eqnarray*}
E (u_n) & = & E^0 (\phi_{\lambda, n}) + o (1) \ = \ E^0 (\xi_n) + o (1) \\
& \geq & E^0 (S_\mu) + o (1), \qquad\text{as}\quad n \to \infty,
\end{eqnarray*}
where $S_\mu$ is a ground state for $E^0$ at mass $\mu$.
So, passing to the limit,
$$ \mathcal E (\mu) \ \geq \mathcal E^0 (\mu), $$
that contradicts Proposition \ref{Eps<Eps0} and then $q_n$ cannot converge to zero. This conclusion holds for every subsequence of a minimizing sequence for $E$, therefore limit points of the complex sequence $q_n$ must be separated from zero, and the proof is complete.
\end{proof}

\begin{lemma}
\label{minimseq2}
Let $p\in(2,4)$, $\alpha\in\R$ and $\mu>0$. Let also $(u_{n})_n$ be a $\delta$-NLS energy minimizing sequence of mass $\mu$. Then, it is bounded in $L^r(\Rd)$, for every $r\geq2$, and there exists $u\in D\setminus H^1(\Rd) $ such that, up to subsequences,
\begin{itemize}
 \item $u_n\deb u$ in $L^2(\Rd)$,
 \item $u_n\to u$ a.e. in $\Rd$,
\end{itemize}
as $n\to+\infty$. In particular, if one fixes $\lambda>0$ and the decomposition $u_{n}=\phi_{n,\la}+q_{n}\G_{\la}$, then $(\phi_{n,\la})_n$ and $(q_n)_n$ are bounded in $H^1(\Rd)$ and $\C$, respectively, and there exist $\phi_{\la}\in H^{1}(\Rd)$ and $q\in \C\setminus\{0\}$ such that $u=\phi_\la+q\G_\la$ and, up to subsequences,
\begin{itemize}
 \item $\phi_{n,\la}\deb\phi_{\la}$ in $L^2(\Rd)$,
 \item $\nabla\phi_{n,\la}\deb\nabla\phi_{\la}$ in $L^2(\Rd)$,
 \item $q_{n}\rightarrow q$ in $\C$,
\end{itemize}
as $n\to+\infty$.
\end{lemma}

\begin{proof}
Let $(u_{n})_n$ be a $\delta$-NLS energy minimizing sequence of mass $\mu$. By Banach-Alaoglu Theorem, $u_n\deb u$ in $L^2(\Rd)$ up to subsequences. Moreover, owing to Lemma \ref{q>C} we can suppose without loss of generality that for every $n$ the charge $q_n$ satisfies $|q_n| > C > 0$, then we can rely on the decomposition introduced in \eqref{Sing} and used in \eqref{GNDs}, namely $u_{n}=\phi_{n}+q_{n}\G_{\nu_{n}}$ with $\nu_{n}:=\f{|q_{n}|^{2}}{\|u_{n}\|_{2}^{2}}$. This decomposition guarantees 
$$ \| \phi_n \|_2 \leq \|u_{n}\|_{2}+|q_{n}|\|\G_{\nu_{n}}\|_{2}\leq \left(1+\f{1}{2\sqrt{\pi}}\right) \sqrt{\mu} $$
for every $n$,
so that the sequence $\phi_n$ is bounded in $L^2 (\R^2)$. 

Using \eqref{Epala} and \eqref{Eun-prel}, we have 
\begin{equation}
\label{Eun}
\begin{split}
E(u_{n})&\geq\left (\f{1}{2}\|\na\phi_{n}\|^2_{2}-\f{M_{p}}{p}\mu\|\na\phi_{n}\|^{p-2}_{2}\right) +\f{|q_{n}|^{2}\|\phi_{n}\|^2_{2}}{2\mu}\\
&
+\left(\alpha-1+\f{\log\left(\f{|q_{n}|}{2\sqrt{\mu}}+\gamma\right)}{2\pi}\right)\f{|q_{n}|^2}{2}-\f{M_{p}}{p}\mu|q_{n}|^{p-2},
\end{split}
\end{equation}
for a suitable $M_p>0$. First, we note by \eqref{Eun} that $(\nabla\phi_n)_n$ is bounded in $L^2(\Rd)$ and $(q_n)_n$ is bounded in $\C$, so that, up to subsequences, $q_n\to q$ and $q \neq 0$ since $|q_n| > C > 0$. 

Fix $\la\geq\f{C_{2}}{\mu}$ with $C_2=1+\sup_{n}|q_n|$ and consider the decomposition of each $u_{n}$ according to $\la$, that is $u_{n}=\phi_{n,\la}+q_{n}\G_{\la}$ with $\phi_{\la,n}:=\phi_{n}+q_{n}(\G_{\nu_{n}}-\G_{\la})$. Exploiting \eqref{eq-Glanu} and \eqref{eq-DGlanu} and the estimates on $\phi_n$ and $q_n$, one finds that there exists $M_{1}, M_{2}>0$ such that for every $n\geq\bar{n}$
\begin{equation}
\|\phi_{n,\la}\|_{2}^{2}\le 2\left[\|\phi_{n}\|_{2}^{2}+ \f{1}{4\pi}\left(\f{|q_{n}|^{2}}{\la}+\mu+2|q_{n}|^{2}\f{\log{\la}+\log(\mu)-2\log(|q_{n}|)}{\nu_n-\la}\right)\right]\le M_{1}
\end{equation}
and
\begin{equation}
\|\na\phi_{n,\la}\|_{2}^{2}\le 2\left[\|\na\phi_{n}\|_{2}^{2}+ \f{|q_{n}|^{2}}{4\pi}\left(\left(\la+\f{|q_{n}|^{2}}{\mu}\right)\f{\log{\la}+\log(\mu)-2\log(|q_{n}|)}{\la-\nu_{n}}-2\right)\right]\le M_{2}.
\end{equation}
Hence $(\phi_{n,\la})_n,\,(\nabla\phi_{n,\la})_n$ are bounded in $L^2(\Rd)$, which implies, via the Banach-Alaoglu theorem, that $\phi_{n,\la}\deb\phi_{\la},\,\nabla\phi_{n,\la}\deb\nabla\phi_{\la}$ in $L^{2}(\Rd)$, up to subsequences, and that $u=\phi_\la+q\G_\la$. Furthermore, by Rellich-Kondrakov theorem, $\phi_{n,\la}\to\phi_\la$ in $L^r_{loc}(\Rd)$, for every $r>2$, so that $u_n\to u$ a.e. in $\Rd$.

It is then left to prove that $(\phi_{n,\la})_n,\,(\nabla\phi_{n,\la})_n$ are bounded in $L^2(\Rd)$ also when the decomposition parameter is smaller than $\f{C_2}{\mu}$. To this aim, let $0<\widetilde{\la}< \f{C_{2}}{\mu}$. We can use the decomposition $u_{n}=\phi_{n,\widetilde{\la}}+q_{n}\G_{\widetilde{\la}}$, where $\phi_{n,\widetilde{\la}}=\phi_{n,\la}+q_{n}(\G_{\la}-\G_{\widetilde{\la}})$, with $\la\ge\f{C_2}{\mu}$. However, arguing as before, one can see that $q_{n}(\G_{\la}-\G_{\widetilde{\la}})$ is bounded in $H^1(\Rd)$, which concludes the proof.
\end{proof}

Finally, we have all the tools to prove the existence part of Theorem \ref{exchar-gs}

\begin{proof}[Proof of Theorem \ref{exchar-gs}-(i)]
Let $(u_{n})_n$ be a $\delta$-NLS energy minimizing sequence of mass $\mu$. Assume also, without loss of generality, that it is a subset of $D_\mu\setminus H^1(\Rd)$, so that we can write $u_{n}=\phi_{n,\la}+q_{n}\G_{\la}$, with $q_{n}\neq 0$ and $\la>0$.  As a consequence, all the results of Lemma \ref{minimseq2} hold and all the following limits hold up to subsequences.

Set $m:=\|u\|_2^2$. By weak lower semicontinuity of the $L^2(\Rd)$-norm, $m\le\mu$. Moreover, as $q\neq 0$, $m\neq0$. Assume, then, by contradiction, that $0<m<\mu$. Note that, since $u_n\deb u$ in $L^2(\Rd)$, $\|u_n-u\|_2^2=\mu-m+o(1)$, as $n\to+\infty$. On the one hand, since $p>2$ and $\f{\mu}{\|u_n-u\|_2^2}>1$ for $n$ sufficiently large, there results that
\begin{equation*}
\begin{split}
\Eps(\mu)&\le E\left(\sqrt{\f{\mu}{\|u_n-u\|_2^2}}(u_n-u)\right)=\f{1}{2}\f{\mu}{\|u_n-u\|_2^2}Q(u_n-u)-\f{1}{p}\left(\f{\mu}{\|u_n-u\|_2^2}\right)^{\f{p}{2}}\|u_n-u\|_p^p\\
&<\f{\mu}{\|u_n-u\|_2^2}E(u_n-u)
\end{split}
\end{equation*}
and thus
\begin{equation}
\label{E-un-u}
\liminf_n E(u_n-u)\ge \f{\mu-m}{\mu}\Eps(\mu).
\end{equation}
On the other hand, a similar computation yields 
\begin{equation*}
\Eps(\mu)\le E\left(\sqrt{\f{\mu}{\|u\|_2^2}}u\right)<\f{\mu}{\|u\|_2^2} E(u),
\end{equation*}
so that
\begin{equation}
\label{E-u}
E(u)>\f{m}{\mu}\Eps(\mu).
\end{equation}
In addition, we can also prove that
\begin{equation}
\label{BLext}
E(u_n)=E(u_n-u)+E(u)+o(1)\qquad\text{as} \quad n\to +\infty
\end{equation}
Indeed, since, $u_n\deb u,\,\phi_{n,\la}\deb\phi_{\la},\,\na\phi_{n,\la}\deb\na\phi_{\la}$ in $L^2(\Rd)$ and $q_n\to q$, we have that 
\begin{equation*}
Q(u_n-u)=Q(u_n)-Q(u)+o(1),\qquad n\to +\infty,
\end{equation*}
while $\|u_n\|_p^p\leq C$ and $u_n\to u$ a.e. on $\Rd$, enable one to use the well known Brezis-Lieb lemma (\cite{BL-83}) in order to get
\begin{equation*}
\|u_n\|_p^p=\|u_n-u\|_p^p+\|u\|_p^p+o(1),\qquad n\to +\infty.
\end{equation*}
Combining \eqref{E-un-u}, \eqref{E-u} and \eqref{BLext}, one can see that
\begin{equation*}
\Eps(\mu)=\liminf_n E(u_n)=\liminf_n E(u_n-u)+ E(u)>\f{\mu-m}{\mu}\Eps(\mu)+\f{m}{\mu}\Eps(\mu)=\Eps(\mu),
\end{equation*}
which is a contradiction. Therefore, $m=\mu$, so that $u\in D_\mu$ and, in particular, $u_n\to u$ in $L^2(\Rd)$ and $\phi_{n,\la}\to \phi_\la$ in $L^2(\Rd)$. 

It is, then, left to show that 
\begin{equation}
\label{eq-semicont}
E(u)\le \liminf_n E(u_n)=\Eps(\mu).
\end{equation}
However, by all the limits obtained before, it suffices to prove that $u_n\to u$ in $L^p(\Rd)$, in order to get \eqref{eq-semicont}. Now, from \eqref{GND},
\begin{equation*}
\|u_n-u\|_p^p\le K_p \left(\|\na\phi_{n,\la}-\na \phi_\la\|_{2}^{p-2}\|\phi_{n,\la}-\phi_\la\|_{2}^{2}+\f{|q_n-q|^{p}}{\la}\right)
\end{equation*}
and then since $\|\na\phi_{n,\la}-\na \phi_\la\|_{2}$ is bounded, $\phi_{n,\la}\to \phi_\la$ in $L^2(\Rd)$ and $q_n\to q$ in $\C$, the claim is proved.
\end{proof}

%%%%%%%%%%%%%%%%%%%%%%%%%%%%%%%%%%%%%%%%%%%%%%%%%%%%%%%%%%%%%%%%%%%%%%%%%%%%%%%%
%%%%%%%%%%%%%%%%%%%%%%%%%%%%%%%%%%%%%%%%%%%%%%%%%%%%%%%%%%%%%%%%%%%%%%%%%%%%%%%%
%%%%%%%%%%%%%%%%%%%%%%%%%%%%%%%%%%%%%%%%%%%%%%%%%%%%%%%%%%%%%%%%%%%%%%%%%%%%%%%%
%%%%%%%%%%%%%%%%%%%%%%%%%%%%%%%%%%%%%%%%%%%%%%%%%%%%%%%%%%%%%%%%%%%%%%%%%%%%%%%%
%%%%%%%%%%%%%%%%%%%%%%%%%%%%%%%%%%%%%%%%%%%%%%%%%%%%%%%%%%%%%%%%%%%%%%%%%%%%%%%%
%%%%%%%%%%%%%%%%%%%%%%%%%%%%%%%%%%%%%%%%%%%%%%%%%%%%%%%%%%%%%%%%%%%%%%%%%%%%%%%%
%%%%%%%%%%%%%%%%%%%%%%%%%%%%%%%%%%%%%%%%%%%%%%%%%%%%%%%%%%%%%%%%%%%%%%%%%%%%%%%%
%%%%%%%%%%%%%%%%%%%%%%%%%%%%%%%%%%%%%%%%%%%%%%%%%%%%%%%%%%%%%%%%%%%%%%%%%%%%%%%%
%%%%%%%%%%%%%%%%%%%%%%%%%%%%%%%%%%%%%%%%%%%%%%%%%%%%%%%%%%%%%%%%%%%%%%%%%%%%%%%%
%%%%%%%%%%%%%%%%%%%%%%%%%%%%%%%%%%%%%%%%%%%%%%%%%%%%%%%%%%%%%%%%%%%%%%%%%%%%%%%%
%%%%%%%%%%%%%%%%%%%%%%%%%%%%%%%%%%%%%%%%%%%%%%%%%%%%%%%%%%%%%%%%%%%%%%%%%%%%%%%%
%%%%%%%%%%%%%%%%%%%%%%%%%%%%%%%%%%%%%%%%%%%%%%%%%%%%%%%%%%%%%%%%%%%%%%%%%%%%%%%%

%%%%%%%%%%%%%%%%%%%%%%%%%%%%%%%%%%
%%%%%%% Action minimizers %%%%%%%%
%%%%%%%%%%%%%%%%%%%%%%%%%%%%%%%%%%

\section{Action minimizers existence: proof of Theorem \ref{exchar-actmin} -- (i)}
\label{sec-am}

The aim of this section is proving point $(i)$ of Theorem \ref{exchar-actmin}, that is the existence/nonexistence of $\delta$-NLS action minimizers at frequency $\omega$.

Preliminarily, we recall that, in the standard case, NLS-action minimizers are those functions $u\in N_\omega^0$ such that $S_\omega^0(v)=d^0(\omega)$, with
\begin{gather*}
d^0(\omega):=\inf_{v\in N^0_\omega}S_\omega^0(v),\\
 S_\omega^0(v):=E^0(v)+\f{\omega}{2}\|v\|_2^2,\\
 N_\omega^0:=\{v\in H^1(\Rd)\setminus\{0\}:I_\omega^0(v)=0\},\qquad I_\omega^0(v):=\|\na v\|_{2}^{2}+\omega\|v\|_{2}^{2}-\|v\|_{p}^{p}.\\
\end{gather*}
We also note that
\begin{equation}
\label{eq-funceq}
 S_\omega(v)=\widetilde{S}(v)>0,\qquad\forall v\in N_\omega,
\end{equation}
with $S_\omega$ and $N_\omega$ given by \eqref{SwE} and \eqref{eq-nehari}, respectively, and
\begin{equation*}
\widetilde{S}(v):=\f{p-2}{2p}\|v\|_p^p.
\end{equation*}
Hence, combining with the fact that ${S_\omega}_{|_{H^1(\Rd)}}=S_\omega^0$ and $N_\omega\cap H^1(\Rd)=N_\omega^0$, it is straightforward that 
\begin{equation}
\label{0dwd0w}
0\le d(\omega)\le d^0(\omega),\qquad\forall \omega\in\R.
\end{equation}
In addition, since $d^0(\omega)=0$, for every $\omega\leq0$ (see, e.g., \cite[Lemma 2.4 and Remark 2.5]{DST-21}), one immediately sees that $d(\omega)=0$, for every $\omega\le0$, which entails that there cannot be any $\delta$-NLS action minimizer at frequency $\omega$ whenever $\omega\leq0$. In view of this we will focus throughout only on the case $\omega>0$. 
 
Now, the first step of our discussion is to detect for which $\omega>0$ the two inequalities in \eqref{0dwd0w} are strict. To this aim let us introduce the set 
\begin{equation*}
\widehat{N}_{\omega}:=\{q\G_{\la}: \la>0,\,q\in\C\setminus\{0\},\,I_{\omega}(q\G_{\la})=0\},
\end{equation*}
which is the subset of $N_\omega$ containing those functions admitting a decomposition with the sole singular part for at least one value of $\lambda>0$. The next two lemmas characterize the set $\widehat{N}_{\omega}$ on varying $\omega>0$.

\begin{lemma}
\label{qlacond}
Let $p>2$, $\alpha\in\R$ and $\omega>0$. Then, $q\G_{\la}\in \widehat{N}_{\omega}$ if and only if $\la>0$ and $q\in \C\setminus\{0\}$ satisfy
\begin{equation}
\label{admla}
\f{\omega-\la}{4\pi}+\la\left(\alpha+\theta_\la\right)>0
\end{equation}
(with $\theta_\la$ defined by \eqref{eq-thla}) and
\begin{equation}
\label{qandla}
|q|=\f{1}{K_p}\left[\f{\omega-\la}{4\pi}+\la\left(\alpha+\theta_\la\right)\right]^{\f{1}{p-2}},
\end{equation}
with $K_p=\|\G_1\|_p^{\f{p}{p-2}}$.

\end{lemma}
\begin{proof}
Fix $\psi=q\G_{\la}$ with $q\neq 0$ and $\la>0$. By \eqref{eq-GLP}, $I_{\omega}(\psi)=0$ if and only if
\begin{equation*}
\f{\omega-\la}{4\pi\la}|q|^{2}+\left(\alpha+\theta_\la\right)|q|^{2}-\f{K}{\la}|q|^{p}=0, 
\end{equation*}
with $K:=\|\G_1\|_p^p$, which entails
\begin{equation}
|q|^{p-2}=\f{1}{K}\left[\f{\omega-\la}{4\pi}+\la\left(\alpha+\theta_\la\right)\right].
\end{equation}
Since $|q|^{p-2}>0$, \eqref{admla} and \eqref{qandla} follow.
\end{proof}

Let us define, now,
\begin{equation}
\label{omega0}
\omega_{0}:=-\ell_\alpha,
\end{equation}
with $\ell_\al$ defined by \eqref{eq-spectrum}.
\begin{lemma}

\label{Nomega}
Let $p>2$, $\alpha\in\R$, $\omega>0$ and $\omega_0$ as in \eqref{omega0}. Therefore:
\begin{itemize}
 \item[(i)] if $\omega\in(0,\omega_{0})$, then \begin{equation}
\widehat{N}_{\omega}=\{q\G_{\la}: \la\in(0,\la_{1}(\omega))\cup(\la_{2}(\omega),+\infty),\,q\in\C\setminus\{0\}\text{ and satisfies \eqref{qandla}}\},
\end{equation}
with $\la_{1}(\omega)\in(0,\omega_{0})$ and $\la_{2}(\omega)>\omega_{0}$ the sole solutions of the equation
\begin{equation*}
\f{\omega-\la}{4\pi}+\la\left(\alpha+\theta_\la\right)=0;
\end{equation*}
 \item[(ii)] if $\omega=\omega_{0}$, then
\begin{equation}
\widehat{N}_{\omega}=\{q\G_{\la}:\la>0,\,\la\neq \omega_{0},\,q\in\C\setminus\{0\}\text{ and satisfies \eqref{qandla}}\};
\end{equation}
\item[(iii)] if $\omega>\omega_{0}$, then 
\begin{equation}
\widehat{N}_{\omega}=\{q\G_{\la}:\la>0,\,q\in\C\setminus\{0\}\text{ and satisfies \eqref{qandla}}\}
\end{equation}
\end{itemize}
\end{lemma}

\begin{proof}
Let $\omega>0$ and introduce the function 
\begin{equation*}
g(\la):=\f{\omega-\la}{4\pi}+\la\left(\alpha+\theta_\la\right).
\end{equation*}
Recall that, in view of Lemma \ref{qlacond}, $q\G_{\la}\in \widehat{N}_{\omega}$ if and only if $g(\la)>0$ and $q$ satisfies \eqref{qandla}, namely $|q|=K_p^{-1}g^{\f{1}{p-2}}(\la)$. Now, it is straightforward (by \eqref{eq-thla}) that 
\begin{equation*}
\lim_{\la\to 0^{+}}g(\la)=\f{\omega}{4\pi}>0\,,\quad \lim_{\la\to +\infty} g(\la)=+\infty
\end{equation*}
and
\begin{equation*}
g'(\la)=\alpha+\theta_\la.
\end{equation*}
Hence, one can easily see that $g$ is decreasing for $\la<\omega_{0}$ and increasing for $\la>\omega_{0}$, has a global minimizer at $\la=\omega_{0}$ and $g(\omega_{0})=\f{\omega-\omega_{0}}{4\pi}$. Therefore, if $\omega>\omega_{0}$, then condition \eqref{admla} can be satisfied for every $\la>0$. On the contrary, if $\omega=\omega_{0}$, then \eqref{admla} can be satisfied provided that $\la>0$ and $\la\neq\omega_{0}$. Finally, if $\omega<\omega_{0}$, then $g(\omega_{0})<0$ and this implies that there exist $\la_{1}(\omega),\,\la_{2}(\omega)>0$ such that \eqref{admla} does not hold if and only if $\la\in [\la_{1}(\omega),\la_{2}(\omega)]$. Note that $\la_{1}(\omega)$ and $\la_{2}(\omega)$ are the only values of $\la>0$ for which $g$ vanishes.
\end{proof}

After this characterization of the set $\widehat{N}_\omega$, we can estimate the value of $d(\omega)$ for $\omega\in(0,\omega_0]$.

\begin{proposition}
\label{pro-dzero}
Let $p>2$, $\alpha\in\R$. Then, $d(\omega)=0$, for every  $\omega\in(0,\omega_{0}]$. 
\end{proposition}

\begin{proof}
Let us discuss separately the cases $\omega\in(0,\omega_{0})$ and $\omega=\omega_{0}$. If $\omega\in(0,\omega_{0})$, then in view of Lemma \ref{Nomega} one can check that 
\begin{equation*}
\lim_{\substack{\la\to\la_{1}(\omega)^{-},\\ q\G_{\la}\in N_{\omega}}}|q|=\lim_{\la\to\la_{1}(\omega)^{-}}\f{1}{K_p}\left[\f{\omega-\la}{4\pi}+\la\left(\alpha+\theta_\la\right)\right]^{\f{1}{p-2}}=0.
\end{equation*}
Hence, recalling and \eqref{eq-funceq} and \eqref{eq-GLP},
\begin{multline*}
0\le d(\omega)\le \inf_{q\G_{\la}\in N_{\omega}}S_{\omega}(q\G_{\la})\leq\lim_{\substack{\la\to\la_{1}(\omega)^{-},\\ q\G_{\la}\in N_{\omega}}}S_{\omega}(q\G_{\la})\\
=\lim_{\substack{\la\to\la_{1}(\omega)^{-},\\ q\G_{\la}\in N_{\omega}}}\widetilde{S}(q\G_{\la})=\lim_{\substack{\la\to\la_{1}(\omega)^{-},\\ q\G_{\la}\in N_{\omega}}}\f{p-2}{2p}\|\G_{1}\|_{p}^{p}\f{|q|^{p}}{\la}=0.
\end{multline*}
If, on the contrary, $\omega=\omega_{0}$, then one obtains the same result, just arguing as before and replacing the limits for $\la\to\la_{1}(\omega)^{-}$ with the limits for $\la\to\omega_{0}$.
\end{proof}

This result has an immediate consequence on the existence of the $\delta$-NLS action minimizers below $\omega_0$.

\begin{corollary}
\label{noexistlowfreq}
Let $p>2$, $\alpha\in\R$. Then, there exists no $\delta$-NLS action minimizer at frequency $\omega$, for every $\omega\in(0,\omega_0]$.
\end{corollary}
\begin{proof}
The claim follows by Proposition \ref{pro-dzero} and \eqref{eq-funceq}.
\end{proof}

On the other hand, in order to discuss the behavior of $d(\omega)$ when $\omega>\omega_0$, it is preliminarily necessary to further investigate the relation between $S_\omega$ and $\widetilde{S}$.

\begin{lemma}
\label{equivprob}
Let $p>2$, $\alpha\in\R$ and $\omega>\omega_0$. Then
\begin{equation}
\label{eq-firstequiv}
d(\omega)=\inf_{v\in\widetilde{N}_\omega}\widetilde{S}(v),
\end{equation}
with
\[
 \widetilde{N}_\omega:=\{v\in D\setminus\{0\}:I_{\omega}(v)\le0\}
\]
(and $I_\omega$ defined by \eqref{eq-Inehari}). Moreover, for any function $u\in D\setminus\{0\}$,
\begin{equation}
\label{eq-secondequiv}
\left\{\begin{array}{l}\displaystyle\widetilde{S}(u)=d(\omega)\\[.2cm]\displaystyle I_{\omega}(u)\le 0\end{array}\right.\qquad\Longleftrightarrow\qquad \left\{\begin{array}{l}\displaystyle S_{\omega}(u)=d(\omega)\\[.2cm]I_{\omega}(u)=0.\end{array}\right.
\end{equation}
\end{lemma}

\begin{remark}
\label{rem-peq}
In view of this lemma, searching for $\delta$-NLS action minimizers is equivalent to searching for
\[
 u\in\widetilde{N}_\omega\qquad\text{such that}\qquad\widetilde{S}(u)=\inf_{v\in\widetilde{N}_\omega}\widetilde{S}(v)=d(\omega).
\]
\end{remark}

\begin{proof}[Proof of Lemma \ref{equivprob}]
We divide in proof in two parts.

\emph{Part (i): proof of \eqref{eq-firstequiv}}. On the one hand, if $u\in N_{\omega}$, then $S_{\omega}(u)=\widetilde{S}(u)$, so that
\begin{equation*}
\inf_{v\in\widetilde{N}_\omega}\widetilde{S}(v)\le d(\omega),
\end{equation*}
as $\widetilde{N}_\omega\supset N_\omega$. On the other hand, fix $u\in D\setminus\{0\}$ such that $I_{\omega}(u)<0$ (i.e. $u\in \widetilde{N}_\omega\setminus N_\omega$). Now, for any fixed $\beta>0$
\begin{equation*}
I_{\omega}(\beta u)=\beta^{2}Q_\omega(u)-\beta^{p}\|u\|_{p}^{p},
\end{equation*}
(see \eqref{Qomega} for the definition of $Q_\omega$), and thus $I_{\omega}(\beta u)=0$ (i.e. $u\in N_\omega$) if and only
\begin{equation*}
\beta=\beta(u):=\bigg(\f{Q_\omega(u)}{\|u\|_{p}^{p}}\bigg)^{\f{1}{p-2}}
\end{equation*}
(where we also used that $Q_\omega(u)>0$, for every $u\in D\setminus\{0\}$, whenever $\omega>\omega_0$). Moreover, since $I_{\omega}(u)<0$, $\beta(u)<1$ and hence
\begin{equation*}
S_{\omega}(\beta(u)u)=\widetilde{S}(\beta(u)u)=\beta(u)^{p}\widetilde{S}(u)<\widetilde{S}(u).
\end{equation*}
As a consequence
\begin{equation*}
 d(\omega)\le\inf_{v\in\widetilde{N}_\omega}\widetilde{S}(v),
\end{equation*}
which completes the proof.

\emph{Part (ii): proof of \eqref{eq-secondequiv}}. If $u\in N_{\omega}$ and $S_{\omega}(u)=d(\omega)$, then clearly $u\in\widetilde{N}_\omega$ and (by \eqref{eq-funceq}) $\widetilde{S}(u)=d(\omega)$. On the contrary, assume by contradiction that $u\in \widetilde{N}_\omega\setminus N_\omega$. If $\widetilde{S}(u)=d(\omega)$, then, arguing as before, one obtains that $\beta(u)u\in N_\omega$ and
\begin{equation*}
S_{\omega}(\beta(u)u)<d(\omega),
\end{equation*}
which is impossible. Hence, if $u\in\widetilde{N}_\omega$ and $\widetilde{S}(u)=d(\omega)$, then $u\in N_{\omega}$ and $S_{\omega}(u)=d(\omega)$.
\end{proof}

We can now prove that the left inequality of \eqref{0dwd0w} is strict.

\begin{proposition}
\label{dpos}
Let $p>2$, $\alpha\in\R$. Then, $d(\omega)>0$, for every  $\omega>\omega_{0}$.
\end{proposition}

\begin{proof}
First, let $u\in \widetilde{N}_\omega\cap H^1(\Rd)$. By Sobolev inequality, for any $p\in(1,+\infty)$ there exists $C_p>0$, depending only on $p$, such that
\begin{equation*}
0\ge I_{\omega}(u)=\|\na u\|_{2}^{2}+\omega\|u\|_{2}^{2}-\|u\|_{p}^{p}\ge C_p\|u\|_{p}^{2}+\omega\|u\|_{2}^{2}-\|u\|_{p}^{p}\ge C_p\|u\|_{p}^{2}-\|u\|_{p}^{p}.
\end{equation*}
Hence, $\|u\|_{p}^{p-2}\ge C_p$ and so
\begin{equation}
\label{Spos}
\widetilde{S}(u)\ge \f{p-2}{2p} C_p^{\f{p}{p-2}},
\end{equation}
whence
\begin{equation}
\label{eq-infH1}
 \inf_{v\in\widetilde{N}_\omega\cap H^1(\Rd)}\widetilde{S}(v)\ge\f{p-2}{2p} C_p^{\f{p}{p-2}}>0
\end{equation}
Consider now a function $u=\phi_\la+q\G_{\la}\in \widetilde{N}_\omega\setminus H^1(\Rd)$ (so that $q\neq 0$) and fix $\la\in(\omega_{0},\omega)$. Clearly $\left(\alpha+\theta_\la\right)>0$, and thus there exists a constant $C>0$ such that
\begin{equation}
\label{Lpartest}
\|\na\phi_{\la}\|_{2}^{2}+\la\|\phi_{\la}\|_{2}^{2}+(\omega-\la)\|u\|_{2}^{2}+|q|^{2}\left(\alpha+\theta_\la\right)\ge C\left(\|\phi_{\la}\|_{H^{1}}^{2}+|q|^{2}\right).
\end{equation}
Moreover, by Sobolev inequality we have that
\begin{equation*}
\|u\|_{p}^{p}\le C_{p}\left(\|\phi_{\la}\|_{p}^{p}+|q|^{p}\right)\le C_{p}\left(\|\phi_{\la}\|_{H^{1}}^{p}+|q|^{p}\right)\le C_{p}\left(\|\phi_{\la}\|_{H^{1}}^{2}+|q|^{2}\right)^{\f{p}{2}},
\end{equation*}
which implies
\begin{equation}
\label{up2}
\|\phi_{\la}\|_{H^{1}}^{2}+|q|^{2}\ge\f{1}{C_{p}}\|u\|_{p}^{2}.
\end{equation}
Then, combining \eqref{Lpartest} and \eqref{up2},
\begin{equation*}
0\ge I_{\omega}(u)\ge C\left(\|\phi_{\la}\|_{H^{1}}^{2}+|q|^{2}\right)-\|u\|_{p}^{p}\ge \f{C}{C_{p}}\|u\|_{p}^{2}-\|u\|_{p}^{p}
\end{equation*}
and so, arguing as before, there exists $K_p>0$, depending only on $p$, such that
\begin{equation*}
\widetilde{S}(u)\ge K_p
\end{equation*}
and, consequently,
\begin{equation}
\label{infD}
\inf_{v\in\widetilde{N}_\omega\setminus H^1(\Rd)}\widetilde{S}(v)\ge K_p>0.
\end{equation}
Finally, combining \eqref{eq-infH1} and \eqref{infD}, we obtain the claim.
\end{proof}

For what concerns the right inequality in \eqref{0dwd0w}, we need to recall preliminarily some of the main properties of the NLS action minimizers at frequency $\omega$, that is functions $u\in N_\omega^0$ such that $S_{\omega}^{0}(u)=d^{0}(\omega)$ (see Theorem 8.1.5 in \cite{C-CL03}).

\begin{proposition}
\label{exuniqS0}
Let $p>2$ and $\omega>0$. Then, there exists at least an NLS action minimizer at frequency $\omega$. In particular, such minimizer $u$ is unique, positive and radially symmetric decreasing, up to gauge and translations invariances.
\end{proposition}

Then, we can prove that also the right inequality of \eqref{0dwd0w} is strict.

\begin{proposition}
\label{comparinf}
Let $p>2$, $\alpha\in\R$. Then, $d(\omega)<d^{0}(\omega)$, for every  $\omega>\omega_{0}$.
\end{proposition}

\begin{proof}
For a fixed $\omega>\omega_0$, let $u$ be the unique positive NLS action minimizer at frequency $\omega$ provided by Proposition \ref{exuniqS0}. Then, $u$ cannot be also a $\delta$-NLS action minimizer at frequency $\omega$. Indeed, if $u$ were a $\delta$-NLS action minimizer at frequency $\omega$, then $u$ would have to satisfy \eqref{eq-regbound} and, in particular, $\phi_\la(\z)=(\alpha+\theta_\la)q$, but this can be proved to be a contradiction with the positivity of $u$ by arguing as in the proof of Proposition \ref{Eps<Eps0}. Hence, there exists $v\in N_{\omega}\setminus H^1(\Rd)$ such that $S_{\omega}(v)<S_{\omega}(u)=d^{0}(\omega)$, which concludes the proof.
\end{proof}

Finally, we have all the tools to prove the existence part of Theorem {\ref{exchar-actmin}}. 

\begin{proof}[Proof of Theorem \ref{exchar-actmin}-(i)]
The case $\omega\leq\omega_0$ has been already proved by the remarks at the beginning of the section and by Corollary \ref{noexistlowfreq}. On the contrary, it is convenient to divide the proof of the case $\omega>\omega_0$ in four steps. We also note that, as in the proof of point (i) of Theorem \ref{exchar-gs}, many of the following limits has to be meant as valid up
to subsequences. We do not repeat it for the sake of simplicity and since this does not give rise to misunderstandings.

\emph{Step 1: weak convergence of the minimizing sequences.} Fix $\omega>\omega_0$ and let $(u_{n})_{n}$ be a \emph{$\delta$-NLS action minimizing sequence at frequency $\omega$}, that is (by Remark \ref{rem-peq}) $(u_{n})_n\subset \widetilde{N}_\omega$ and $\widetilde{S}(u_{n})\to d(\omega)$, as $n\to +\infty$. In addition, for any fixed $\la>0$ we can use for $u_{n}$ the decomposition $u_{n}=\phi_{n,\la}+q_{n}\G_{\la}$. First, we see that, since $\|u_{n}\|_{p}^{p}\to \f{2p}{p-2}d(\omega)$, $(u_{n})_{n}$ is bounded in $L^p(\Rd)$. Moreover, as $I_{\omega}(u_{n})\le 0$, we get
\begin{equation*}
\|\na \phi_{n,\la}\|_{2}^{2}+\la\|\phi_{n,\la}\|_{2}^{2}+(\omega-\la)\|u_{n}\|_{2}^{2}+\left(\alpha+\theta_\la\right)|q_{n}|^{2}\le \|u_{n}\|_{p}^{p}.
\end{equation*}
Now, if one sets $\la=\f{\omega+\omega_{0}}{2}$, then the three constants in front of $\|\phi_{n,\la}\|_{2}^{2}$, $\|u_{n}\|_{2}^{2}$ and $|q_{n}|^{2}$ are all strictly positive. Hence, $(\na \phi_{n,\la})_n$,  $(\phi_{n,\la})_n$ and $(u_{n})_n$ are bounded in $L^2(\Rd)$ and $(q_{n})_n$ is bounded in $\C$. Thus, there exists $\phi_{\la}\in H^{1}(\Rd)$, $q\in \C$ and $u\in D$ such that $u=\phi_{\la}+q\G_{\la}$ and
\begin{equation*}
\nabla\phi_{n,\la}\deb\nabla\phi_{\la},\quad\phi_{n,\la}\deb\phi_{\la}\quad u_{n}\deb u \quad\text{in}\quad L^{2}(\Rd)\qquad \text{and} \qquad q_{n}\to q\quad\text{in}\quad\C.
\end{equation*}

\emph{Step 2: $u\in D\setminus H^{1}(\Rd)$.} Assume, by contradiction, that $u\in H^{1}(\Rd)$, namely that $q=0$, and define the sequence $w_{n}:=\sigma_{n}\phi_{n,\la} \in H^{1}(\Rd)$, with \begin{equation*}
\sigma_{n}:=\left(1+\f{I_{\omega}(u_{n})-\left(\alpha+\theta_\la\right)|q_{n}|^{2}+(\|u_{n}\|_{p}^{p}-\|\phi_{n,\la}\|_{p}^{p})+(\omega-\la)(\|\phi_{n,\la}\|_{2}^{2}-\|u_{n}\|_{2}^{2})}{\|\phi_{n,\la}\|_{p}^{p}}\right)^{\f{1}{p-2}},
\end{equation*}
so that $I^{0}_{\omega}(\sigma_{n}\phi_{n,\la})=0$. Note that $\sigma_{n}$ is well defined since there exists $C>0$ such that $\|\phi_{n,\la}\|_{p}^{p}\ge C$ for every $n\in \N$. Indeed, by Proposition \ref{dpos}, $\|u_{n}\|_{p}^{p}$ is uniformly bounded away from zero and $q_{n}\to 0$. On the other hand, since $|q_{n}|^{2}\to 0$, it follows that both $\|\phi_{n,\la}\|_{2}^{2}-\|u_{n}\|_{2}^{2}\to 0$ and
\begin{equation*}
\left|\|u_{n}\|_{p}^{p}-\|\phi_{n,\la}\|_{p}^{p}\right|\le C_{1}\left|\|u_{n}\|_{p}-\|\phi_{n,\la}\|_{p}\right|\le C_{2}\|u_{n}-\phi_{n,\la}\|_{p}\to 0.
\end{equation*}
As a consequence, since $I_\omega(u_n)\le0$, $(\sigma_n^p)_n$ is bounded from above by a sequence $(a_n)_n$ converging to 1. Thus, as $I^{0}_{\omega}(w_{n})=0$ and $\widetilde{S}(u_{n})\to d(\omega)$,
\begin{equation*}
d^0(\omega)+o(1)=\widetilde{S}(w_{n})=\sigma_{n}^{p}\widetilde{S}\left(\phi_{n,\la}\right)\leq a_{n}\left(\widetilde{S}(u_{n})+o(1)\right)=\widetilde{S}(u_{n})+o(1)=d(\omega)+o(1),
\end{equation*}
that implies that $d(\omega)\ge d^{0}(\omega)$, which contradicts Proposition \ref{comparinf}. 

\emph{Step 3: $u\in \widetilde{N}_\omega$.} In view of Step 2, it is left to prove that $I_{\omega}(u)\le 0$. Assume by contradiction that $I_{\omega}(u)>0$. From boundedness of $\phi_{n,\la}$ in $H^1(\Rd)$ and $q_n\to q$, one sees that $u_n\to u$ in $L^p_{loc}(\Rd)$ and hence $u_n\to u$ a.e. in $\Rd$. As $(u_n)_n$ is bounded in $L^p(\Rd)$, one can use Brezis-Lieb lemma to get $\|u_{n}\|_{p}^{p}-\|u_{n}-u\|_{p}^{p}-\|u\|_{p}^{p}\to 0$, and thus
\begin{equation}
\label{eq-Sbrezis}
\widetilde{S}(u_{n})-\widetilde{S}(u_{n}-u)-\widetilde{S}(u)\to 0.
\end{equation}
Since, in addition, $q_{n}\to q$, $\na\phi_{n,\la}\deb\na\phi_{\la}$, $\phi_{n,\la}\deb\phi_{\la}$ and $u_{n}\deb u$ in $L^{2}(\Rd)$ and $Q_\omega$ is quadratic, one can also check that
\begin{equation}
\label{eq-brexisI}
I_{\omega}(u_{n})-I_{\omega}(u_{n}-u)-I_{\omega}(u)\to 0.
\end{equation}
Let us prove now that $I_{\omega}(u_{n})\to 0$. Assume by contradiction that $I_{\omega}(u_{n})\not\to 0$. As $\|u_{n}\|_{p}^{p}\le C$, for some $C>0$,
\begin{equation*}
-C\le I_{\omega}(u_{n})\le 0.
\end{equation*}
Hence, without loss of generality, we can suppose that $I_{\omega}(u_{n})\to -\beta$, with $\beta>0$. Consider, then, the sequence $v_{n}:=\theta_{n}u_{n}$, with 
\begin{equation*}
\theta_{n}:=\left(1+\f{I_{\omega}(u_{n})}{\|u_{n}\|_{p}^{p}}\right)^{\f{1}{p-2}},
\end{equation*}
so that $I_{\omega}(v_{n})=0$. Thus, an easy computation shows that
\begin{equation*}
\theta_{n}\to l:= \left(1-\f{\beta (p-2)}{2 p d(\omega)}\right)^{\f{1}{p-2}}<1.
\end{equation*}
As a consequence,
\begin{equation*}
\widetilde{S}(v_{n})=\widetilde{S}(\theta_{n}u_{n})=\theta_{n}^{p}\widetilde{S}(u_{n})\to l^{p}d(\omega)<d(\omega),
\end{equation*}
which is a contradiction. Hence $I_{\omega}(u_{n})\to 0$. Finally, looking back at \eqref{eq-brexisI}, since $I_{\omega}(u)>0$ and $I_{\omega}(u_{n})\to 0$, 
\begin{equation*}
I_{\omega}(u_{n}-u)=I_{\omega}(u_{n})-I_{\omega}(u)+o(1)=-I_{\omega}(u)+o(1),
\end{equation*}
entailing that $I_{\omega}(u_{n}-u)\to-I_{\omega}(u)<0$. Choose, then, $\bar{n}$ such that $I_{\omega}(u_{n}-u)<0$ for every $n\ge\bar{n}$. Since $d(\omega)\le \widetilde{S}(u_{n}-u)$ and $\widetilde{S}(u)>0$, \eqref{eq-Sbrezis} yields
\begin{equation*}
d(\omega)\leq\lim_{n}\widetilde{S}(u_{n}-u)= d(\omega)-\widetilde{S}(u)<d(\omega),
\end{equation*}
which is again a contradiction and entails $I_\omega(u)\leq0$.

\emph{Step 4: conclusion.} As boundedness in $L^p(\Rd)$ entails that $u_n\deb u $ in $L^p(\Rd)$, by weak lower semicontinuity
\begin{equation*}
\widetilde{S}(u)\le\liminf_{n\to+\infty} \widetilde{S}(u_{n})=d(\omega),
\end{equation*}
which concludes the proof.
 \end{proof}
 
%%%%%%%%%%%%%%%%%%%%%%%%%%%%%%%%%%%%%%%%%%%%%%%%%%%%%%%%%%%%%%%%%%%%%%%%%%%%%%%%
%%%%%%%%%%%%%%%%%%%%%%%%%%%%%%%%%%%%%%%%%%%%%%%%%%%%%%%%%%%%%%%%%%%%%%%%%%%%%%%%
%%%%%%%%%%%%%%%%%%%%%%%%%%%%%%%%%%%%%%%%%%%%%%%%%%%%%%%%%%%%%%%%%%%%%%%%%%%%%%%%
%%%%%%%%%%%%%%%%%%%%%%%%%%%%%%%%%%%%%%%%%%%%%%%%%%%%%%%%%%%%%%%%%%%%%%%%%%%%%%%%
%%%%%%%%%%%%%%%%%%%%%%%%%%%%%%%%%%%%%%%%%%%%%%%%%%%%%%%%%%%%%%%%%%%%%%%%%%%%%%%%
%%%%%%%%%%%%%%%%%%%%%%%%%%%%%%%%%%%%%%%%%%%%%%%%%%%%%%%%%%%%%%%%%%%%%%%%%%%%%%%%
%%%%%%%%%%%%%%%%%%%%%%%%%%%%%%%%%%%%%%%%%%%%%%%%%%%%%%%%%%%%%%%%%%%%%%%%%%%%%%%%
%%%%%%%%%%%%%%%%%%%%%%%%%%%%%%%%%%%%%%%%%%%%%%%%%%%%%%%%%%%%%%%%%%%%%%%%%%%%%%%%
%%%%%%%%%%%%%%%%%%%%%%%%%%%%%%%%%%%%%%%%%%%%%%%%%%%%%%%%%%%%%%%%%%%%%%%%%%%%%%%%
%%%%%%%%%%%%%%%%%%%%%%%%%%%%%%%%%%%%%%%%%%%%%%%%%%%%%%%%%%%%%%%%%%%%%%%%%%%%%%%%
%%%%%%%%%%%%%%%%%%%%%%%%%%%%%%%%%%%%%%%%%%%%%%%%%%%%%%%%%%%%%%%%%%%%%%%%%%%%%%%%
%%%%%%%%%%%%%%%%%%%%%%%%%%%%%%%%%%%%%%%%%%%%%%%%%%%%%%%%%%%%%%%%%%%%%%%%%%%%%%%%

%%%%%%%%%%%%%%%%%%%%%%%%%%%%%%%%%%
%%%%%%% Characterization %%%%%%%%%
%%%%%%%%%%%%%%%%%%%%%%%%%%%%%%%%%%

\section{Further properties: proof of point (ii) of Theorems \ref{exchar-gs} and \ref{exchar-actmin}}
\label{sec-char}

In this section we prove point (ii) in Theorem \ref{exchar-gs} and Theorem \ref{exchar-actmin}, that concern the features of $\delta$-NLS ground states and $\delta$-NLS action minimizers. We point out that, by Lemma \ref{chargs}, proving Theorem \ref{exchar-actmin} implies the conclusion of Theorem \ref{exchar-gs}.

Before proving (ii) of Theorem \ref{exchar-actmin}, let us give an informal description of the strategy.
First, we establish that ground states minimize the functional $Q_\omega$ defined in \eqref{Qomega} on the constraint
\begin{equation}
\label{DComega}
D_{\omega}^p:=\left\{v\in D:\|v\|_{p}^{p}=\f{2p}{p-2}d(\omega)\right\}.
\end{equation}
Second, given a minimizer $u$ of such a problem without the required property (i.e., positivity and radially symmetric monotonicity), we exhibit through rearrangement a function $\widetilde{u}$ such that  
\begin{equation}
\label{Step1}
\|\widetilde{u}\|_p> \|u\|_p\qquad\text{and}\qquad Q_{\omega}(\widetilde{u})\le Q_\omega(u).
\end{equation} 
Moreover, noting that there exists $\beta<1$ such that 
\begin{equation}
\label{Step2}
\|\beta\widetilde{u}\|_{p}^{p}=\|u\|_{p}^{p}\quad\text{and}\quad Q_{\omega}(\beta\widetilde{u})<Q_\omega(\widetilde{u}),
\end{equation}
we find a better competitor with respect to the minimizer, and obtain a contradiction.

\begin{remark}
\label{no-rearr-gs}
Unfortunately, such a strategy is not applicable directly to the minimizers of the energy $E$ or of the action $S_\omega$. More in detail, applying to $E$ the method described above, we obtain
\begin{equation*}
\|\widetilde{u}\|_r> \|u\|_r\quad\text{for every}\quad r\ge2\quad\text{and}\quad E(\widetilde{u})<E(u).
\end{equation*}
However, since the mass constraint is not fulfilled by $\widetilde{u}$, we note that there exists $\beta<1$ such that $\|\beta\widetilde{u}\|_2^2=\|u\|_2^2$, but here, since $\beta^2>\beta^p$ and $E(\widetilde{u})<0$, this yields
\begin{equation*}
E(\beta\widetilde{u})=\f{1}{2}\beta^2 Q(\widetilde{u})-\f{1}{p}\beta^p\|\widetilde{u}\|_p^p>\beta^p E(\widetilde{u})>E(\widetilde{u}),
\end{equation*}
that provides an inequality in the opposite direction with respect to the aimed one.

Analogously, applying the same procedure to the minimization of the action $S_\omega$ on the Nehari manifold, there results 
\begin{equation}
\label{Swu}
I_\omega(\widetilde{u})<0\quad\text{and}\quad S_\omega(\widetilde{u})< S_\omega(u).
\end{equation}
However, if we set 
\begin{equation*}
\bar{\beta}:=\left(\f{Q_\omega(\widetilde{u})}{\|\widetilde{u}\|_p^p}\right)^{\f{1}{p-2}},
\end{equation*}
then $\bar{\beta}<1$, $I_\omega(\bar{\beta}\widetilde{u})=0$ and 
\begin{equation}
\label{Swtildeu}
S_\omega(\bar{\beta}\widetilde{u})=\f{1}{2}\bar{\beta}^2 Q_{\omega}(\widetilde{u})-\f{1}{p}\bar{\beta}^p\|\widetilde{u}\|_p^p>S_\omega(\widetilde{u}).
\end{equation}
Indeed, computing
\begin{equation*}
\f{d}{d\beta}S_\omega(\beta \widetilde{u})=\beta Q_\omega(\tilde{u})-\beta^{p-1}\|\widetilde{u}\|_p^p,
\end{equation*}
we find that $\f{d}{d\beta}S_\omega(\beta \widetilde{u})> 0$ if and only if $0<\beta< \bar{\beta}$, and $\f{d}{d\beta}S_\omega(\beta \widetilde{u})_{|\beta=\bar{\beta}}= 0$, so that $S_\omega(\bar{\beta}\widetilde{u})>S_\omega(\widetilde{u})$. In other words, here again \eqref{Swtildeu} is an inequality in the opposite direction with respect to the aimed one.
\end{remark}

We now start by proving (ii)(a), namely the coexistence of the regular and the singular part for a $\delta$-NLS action minimizer.

\begin{proposition}
\label{GSnornos}
Let $p>2$, $\alpha\in\R$ and $\omega>\omega_{0}$. Let also $u$ be a $\delta$-NLS action minimizer at frequency $\omega$. Then, $q\neq0$ and $\phi_\la:=u-q\G_\la\neq 0$, for every $\la>0$.
\end{proposition}

\begin{proof}
Let $\la>0$ and consider the decomposition $u=\phi_\la+q\G_\la$. Assume by contradiction that $\phi_{\la}=0$. Since $u\neq 0$, clearly $q\neq 0$. As $u$ has to satisfy \eqref{eq-regbound}, then $\alpha+\theta_\la=0$, so that $\la=\omega_{0}$. Since $u$ has to satisfy also \eqref{EL3}, with some computations one obtains that $q$ has to satisfy
\[
 \omega-\omega_0+|q|^{p-2}|\G_{\omega_0}(\x)|^{p-2}=0,\qquad\forall \x\in\Rd\setminus\{\z\},
\]
which is clearly not possible.

On the other hand, assume by contradiction that $q=0$, or equivalently that $u\in H^{1}(\Rd)$. This would imply that $d(\omega)=d^{0}(\omega)$, which contradicts Proposition \ref{comparinf}.
\end{proof}

\begin{remark}
Proposition \ref{GSnornos} marks a difference with the model  \eqref{eq-tNLS_conc}. Indeed, it was proven in \cite{ACCT-20,ACCT-21} that for any bound state there exists a value of $\lambda>0$ such that the regular part of the decomposition vanishes.
\end{remark}

We can move to the proof of point (ii)(b). Preliminarily, we note that, up to the multiplication by a phase factor, a $\delta$-NLS action minimizer $u=\phi_{\la}+q\G_{\la}$ can be assumed to display a charge $q>0$. Indeed, since $\G_{\la}(\x)>0$ for every $\x\in \Rd\setminus\{\z\}$ and $q\neq 0$, it is sufficient to multiply $u$ times $e^{i\theta}$ in such a way that $q e^{i\theta}>0$. In particular, if $\theta$ satisfies the equation $e^{i\theta}=\f{\bar{q}}{|q|}$, then $q e^{i\theta}=|q|$. As a consequence, we will always assume throughout that $q>0$.

The first key point for the proof of (ii)(b) is the switch from the minimization of $S_\omega$ constrained on $N_\omega$ to the minimization of $Q_{\omega}$ constrained on $D_{\omega}^{p}$, which is introduced in the next result. 

\begin{proposition}
\label{equivprob3}
Let $p>2$, $\alpha\in\R$ and $\omega>\omega_{0}$. Then,
\begin{equation}
\inf_{v\in D_{\omega}^p}Q_{\omega}(v)=\f{2p}{p-2}d(\omega),
\end{equation}
with $D_{\omega}^p$ defined in \eqref{DComega},
and there exists a function $u\in D_{\omega}^p$ such that $Q_{\omega}(u)=\f{2p}{p-2}d(\omega)$. In particular, there results that
\begin{equation}
\label{eq-fineq}
 \left\{
 \begin{array}{l}
  \displaystyle Q_{\omega}(w)=\f{2p}{p-2}d(\omega)\\[.2cm]
  \displaystyle w\in D_{\omega}^p
 \end{array}
 \right.
 \qquad\Longleftrightarrow\qquad
 \left\{
 \begin{array}{l}
  \displaystyle S_{\omega}(w)=d(\omega)\\[.2cm]
  \displaystyle w\in N_{\omega}
 \end{array}
 \right.
\end{equation}
\end{proposition}

\begin{remark}
\label{GSequiv}
In view of this result, one sees that, in order to study the features of $\delta$-NLS action minimizers at frequency $\omega$, it is sufficient (in fact, equivalent) to study the minimizers of $Q_\omega$ on $D_{\omega}^p$.
\end{remark}

\begin{proof}[Proof of Proposition \ref{equivprob3}]
Let $u$ be $\delta$-NLS action minimizer at frequency $\omega$. Then, by Lemma \ref{equivprob} $u$ is a minimizer of $\widetilde{S}$ on $\widetilde{N}_\omega$, so that $\|u\|_{p}^{p} \le \|v\|_{p}^{p}$ for every $v\in \widetilde{N}_\omega$, $\|u\|_{p}^{p}=\f{2p}{p-2}d(\omega)$ and $I_{\omega}(u)=0$. 

Let $v\in D_{\omega}^p$. First we see that $I_{\omega}(v)\ge 0=I_{\omega}(u)$. Indeed, if we assume by contradiction that there exists $v\in D\setminus\{0\}$ such that $I_{\omega}(v)<0$, then by Lemma \ref{equivprob} $v$ cannot be a minimizer of $\widetilde{S}$ on $\widetilde{N}_\omega$, and thus $\|v\|_{p}^{p}>\f{2p}{p-2}d(\omega)$, which contradicts the fact that $v\in D_{\omega}^p$. Therefore, $u$ is a minimizer of $I_{\omega}$ on $D_{\omega}^p$, which yields, by using
\begin{equation*}
I_{\omega}(u)=Q_{\omega}(u)-\|u\|_{p}^{p}=Q_{\omega}(u)-\f{2p}{p-2}d(\omega),
\end{equation*}
 that $u$ is also a minimizer of $Q_{\omega}$ on $D_{\omega}^p$ and that $Q_{\omega}(u)=\f{2p}{p-2}d(\omega)$.

This clearly proves the first part of the proposition and the reverse implication in \eqref{eq-fineq}. It is, then, to prove that every minimizer of $Q_{\omega}$ on $D_{\omega}^p$ is a $\delta$-NLS action minimizer at frequency $\omega$. To this aim, let $w$ be a minimizer of $Q_{\omega}$ on $D_{\omega}^p$. It is straightforward that
\begin{equation*}
S_{\omega}(w)=\widetilde{S}(w)=\f{p-2}{2p}\|w\|_{p}^{p}=d(\omega)
\end{equation*}
and, by combining the two equations in \eqref{eq-utile},
\begin{equation*}
I_{\omega}(w)=Q_{\omega}(w)-\|w\|_{p}^{p}=2S_{\omega}(w)-\f{p-2}{p}\|w\|_{p}^{p}=2d(\omega)-2d(\omega)=0,
\end{equation*}
which conclude the proof.
\end{proof}

We can now prove the first part of (ii)(b), which is the positivity up to gauge invariance.

\begin{proposition}
\label{realposgs}
Let $p>2$, $\alpha\in\R$ and $\omega>\omega_{0}$. Then, $\delta$-NLS action minimizers at frequency $\omega$ are positive, up to gauge invariance.
\end{proposition}

\begin{proof}
Let $u$ be a $\delta$-NLS action minimizer at frequency $\omega$. Up to gauge invariance, it is not restrictive to assume $q>0$. In addition, by Proposition \ref{equivprob3}, $u$ is also a minimizer of $Q_{\omega}$ on $D_{\omega}^p$. Now, let us choose $\la=\omega$ in the decomposition of $u$ and define $\Omega:=\{\x\in\Rd:\phi_\omega(\x)\neq0\}$. By Proposition \ref{GSnornos}, $|\Omega|>0$. Then, we can write
\begin{equation*}
u(\x)=\phi_{\omega}(\x)+q\G_{\omega}(\x)=e^{i\eta(\x)}|\phi_{\omega}(\x)|+q\G_{\omega}(\x),\qquad\forall\x\in\Omega\setminus\{\z\},
\end{equation*}
for some $\eta:\Omega\to[0,2\pi)$. If one can prove that $\eta(\x)=0$ for a.e. $\x\in\Omega\setminus\{\z\}$, then the proof is complete as this entails that $\phi_\omega(\x)=|\phi_\omega(\x)|\geq0$ for every $\x\in\Rd$, whence $u(\x)>0$ for every $\x\in\Rd\setminus\{\z\}$.

To this aim, assume by contradiction that $\eta\neq 0$ on $\Omega_1\subset(\Omega\setminus\{\z\})$, with $|\Omega_1|>0$. Letting $\widetilde{u}:=|\phi_{\omega}|+q\G_{\omega}$ (note that $u=\widetilde{u}$ in $\Rd\setminus\Omega_1$), there results that 
\begin{multline*}
|u(\x)|^{2}=|\phi_{\omega}(\x)|^{2}+q^{2}\G_{\omega}^2(\x)+2\cos(\eta(\x))|\phi_{\omega}(\x)|\G_{\omega}(\x)\\[.2cm]
<|\phi_{\omega}(\x)|^{2}+q^{2}\G_{\omega}^2(\x)+2|\phi_{\omega}(\x)|\G_{\omega}(\x)=|\widetilde{u}(\x)|^{2},\qquad\forall \x\in\Omega_1.
\end{multline*}
Hence, as $|\Omega_1|>0$,
\begin{equation}
\label{pnormstrict}
\|u\|_{p}^{p}=\int_{\Rd}\left(|u|^{2}\right)^{\f{p}{2}}\dx<\int_{\Rd}\left(|\widetilde{u}|^{2}\right)^{\f{p}{2}}\dx=\|\widetilde{u}\|_{p}^{p}.
\end{equation}
On the other hand, it is straightforward to check that $Q_{\omega}(\widetilde{u})\le Q_{\omega}(u)$. Now, from \eqref{pnormstrict} and the positivity of $Q_\omega$, there exists $\beta\in(0,1)$ such that $\|\beta\widetilde{u}\|_{p}^{p}=\|u\|_{p}^{p}=\f{2p}{p-2}d(\omega)$ and
\begin{equation*}
Q_{\omega}(\beta\widetilde{u})=\beta^{2}Q_{\omega}(\widetilde{u})<Q_{\omega}(u),
\end{equation*}
which contradicts the fact that $u$ minimizes $Q_\omega$ on $D_{\omega}^p$. Thus $\eta=0$ a.e. on $\Omega\setminus\{\z\}$, which concludes the proof.
\end{proof}

The proof of the previous result also entails that for, $\la=\omega$, the regular part $\phi_\omega$ of a $\delta$-NLS action minimizer at frequency $\omega$ is nonnegative. The following corollary points out that, whenever $\la>\omega$, it is in fact positive.

\begin{corollary}
\label{regpos}
Let $p>2$, $\alpha\in\R$ and $\omega>\omega_{0}$. Let also $u$ be a $\delta$-NLS action minimizer at frequency $\omega$. Then the regular part $\phi_{\la}:=u-q\G_\la$ is positive for every $\la>\omega$, up to gauge invariance.
\end{corollary}

\begin{proof}
Let $u$ be a positive $\delta$-NLS action minimizer at frequency $\omega$ and consider the decomposition $u=\phi_\la+q\G_\la$ for a fixed $\la>\omega$. First, using \eqref{Gla<Gnu} and $q>0$, we see that
\begin{equation*}
\phi_\la(\x)=\phi_\omega(\x)+q(\G_{\omega}(\x)-\G_\la(\x))>0,\qquad\forall \x\in\Rd\setminus\{\z\}.
\end{equation*}
Then, one concludes the proof just recalling \eqref{eq-propG2}.
\end{proof}

Finally, we may address the problem of the radially symmetric monotonicity of $\delta$-NLS action minimizers.
 
\begin{proposition}
\label{minsymm}
Let $p>2$, $\alpha\in\R$ and $\omega>\omega_{0}$. Then, $\delta$-NLS action minimizers at frequency $\omega$ are radially symmetric decreasing, up to gauge invariance.
\end{proposition}

\begin{proof}
Without loss of generality let $u$ be a positive $\delta$-NLS action minimizers at frequency $\omega$. Consider also the decomposition $u=\phi_\omega+q\G_\omega$, corresponding to the choice $\la=\omega$. In order to prove the claim is is sufficient to show that $\phi_\omega=\phi_\omega^*$, with $\phi_\omega^*$ the radially symmetric nonincreasing rearrangement of $\phi_\omega$.

Assume, by contradiction, that $\phi_\omega\neq\phi_\omega^*$, that is $\phi_\omega$ is not radially symmetric nonincreasing. Then, define the function $\widetilde{u}=\phi_\omega^*+q\G_\omega$. By \eqref{PS} and \eqref{equimeas}, we have $\|\na\phi_\omega^*\|_2\le \|\na\phi_\omega\|_2$ and $\|\phi_\omega^*\|_2=\|\phi_\omega\|_2$, so that
\begin{equation*}
Q_\omega(\widetilde{u})\le Q_\omega(u).
\end{equation*}
Now, applying Proposition \ref{fgtheor} with $f=q\G_\omega$ and $g=\phi_\omega$, there results that $\|\widetilde{u}\|_p^p>\|u\|_p^p$, as $\phi_\omega\neq\phi_\omega^*$. Therefore, (as $Q_\omega$ is positive) there exists $\beta<1$ such that $\|\beta \widetilde{u}\|_p^p=\|u\|_p^p$ and
\begin{equation*}
Q_{\omega}(\beta\widetilde{u})=\beta^{2}Q_{\omega}(\widetilde{u})<Q_{\omega}(\widetilde{u})\le Q_{\omega}(u),
\end{equation*}
but, via Proposition \ref{equivprob3} (arguing as in the proof of Proposition \ref{realposgs}), this contradicts that $u$ is a $\delta$-NLS action minimizer, thus concluding the proof.
\end{proof}

We can now sum up all the previous results to prove point (ii) of Theorems \ref{exchar-gs} and \ref{exchar-actmin}. 

\begin{proof}[Proof of Theorems \ref{exchar-gs} and \ref{exchar-actmin}-(ii)]
Let $u$ be a $\delta$-NLS action minimizer at frequency $\omega>\omega_0$. Then, by Proposition \ref{GSnornos}, Proposition \ref{realposgs}, Corollary \ref{regpos} and Proposition \ref{minsymm}, $u$ satisfies all the properties stated in (ii).

Let $p\in(2,4)$ and $u$ be a $\delta$-NLS ground state of mass $\mu$. Combining Lemma \ref{chargs} and point (i) of Theorem \ref{exchar-actmin} one sees that $u$ is also a $\delta$-NLS action minimizer at some frequency $\omega>\omega_0$ (in particular, $\omega=\mu^{-1}(\|u\|_p^p-Q(u))$). Then, one concludes by point (ii) of Theorem \ref{exchar-actmin}.
\end{proof}

%%%%%%%%%%%%%%%%%%%%%%%%%%%%%%%%%%%%%%%%%%%%%%%%%%%%%%%%%%%%%%%%%%%%%%%%%%%%%%%%
%%%%%%%%%%%%%%%%%%%%%%%%%%%%%%%%%%%%%%%%%%%%%%%%%%%%%%%%%%%%%%%%%%%%%%%%%%%%%%%%
%%%%%%%%%%%%%%%%%%%%%%%%%%%%%%%%%%%%%%%%%%%%%%%%%%%%%%%%%%%%%%%%%%%%%%%%%%%%%%%%
%%%%%%%%%%%%%%%%%%%%%%%%%%%%%%%%%%%%%%%%%%%%%%%%%%%%%%%%%%%%%%%%%%%%%%%%%%%%%%%%
%%%%%%%%%%%%%%%%%%%%%%%%%%%%%%%%%%%%%%%%%%%%%%%%%%%%%%%%%%%%%%%%%%%%%%%%%%%%%%%%
%%%%%%%%%%%%%%%%%%%%%%%%%%%%%%%%%%%%%%%%%%%%%%%%%%%%%%%%%%%%%%%%%%%%%%%%%%%%%%%%
%%%%%%%%%%%%%%%%%%%%%%%%%%%%%%%%%%%%%%%%%%%%%%%%%%%%%%%%%%%%%%%%%%%%%%%%%%%%%%%%
%%%%%%%%%%%%%%%%%%%%%%%%%%%%%%%%%%%%%%%%%%%%%%%%%%%%%%%%%%%%%%%%%%%%%%%%%%%%%%%%
%%%%%%%%%%%%%%%%%%%%%%%%%%%%%%%%%%%%%%%%%%%%%%%%%%%%%%%%%%%%%%%%%%%%%%%%%%%%%%%%
%%%%%%%%%%%%%%%%%%%%%%%%%%%%%%%%%%%%%%%%%%%%%%%%%%%%%%%%%%%%%%%%%%%%%%%%%%%%%%%%
%%%%%%%%%%%%%%%%%%%%%%%%%%%%%%%%%%%%%%%%%%%%%%%%%%%%%%%%%%%%%%%%%%%%%%%%%%%%%%%%
%%%%%%%%%%%%%%%%%%%%%%%%%%%%%%%%%%%%%%%%%%%%%%%%%%%%%%%%%%%%%%%%%%%%%%%%%%%%%%%%

%%%%%%%%%%%%%%%%%%%%%%%%%%%%%%%%%%
%%%%%%%%%%% Appendice %%%%%%%%%%%%
%%%%%%%%%%%%%%%%%%%%%%%%%%%%%%%%%%

\appendix

%%%%%%%%%%%%%%%%%%%%%%%%%%%%%%%%%%%%%%%%%%%%%%%%%%%%%%%%%%%%%%%%%%%%%%%%%%%%%%%%
%%%%%%%%%%%%%%%%%%%%%%%%%%%%%%%%%%%%%%%%%%%%%%%%%%%%%%%%%%%%%%%%%%%%%%%%%%%%%%%%
%%%%%%%%%%%%%%%%%%%%%%%%%%%%%%%%%%%%%%%%%%%%%%%%%%%%%%%%%%%%%%%%%%%%%%%%%%%%%%%%
%%%%%%%%%%%%%%%%%%%%%%%%%%%%%%%%%%%%%%%%%%%%%%%%%%%%%%%%%%%%%%%%%%%%%%%%%%%%%%%%
%%%%%%%%%%%%%%%%%%%%%%%%%%%%%%%%%%%%%%%%%%%%%%%%%%%%%%%%%%%%%%%%%%%%%%%%%%%%%%%%
%%%%%%%%%%%%%%%%%%%%%%%%%%%%%%%%%%%%%%%%%%%%%%%%%%%%%%%%%%%%%%%%%%%%%%%%%%%%%%%%
%%%%%%%%%%%%%%%%%%%%%%%%%%%%%%%%%%%%%%%%%%%%%%%%%%%%%%%%%%%%%%%%%%%%%%%%%%%%%%%%
%%%%%%%%%%%%%%%%%%%%%%%%%%%%%%%%%%%%%%%%%%%%%%%%%%%%%%%%%%%%%%%%%%%%%%%%%%%%%%%%
%%%%%%%%%%%%%%%%%%%%%%%%%%%%%%%%%%%%%%%%%%%%%%%%%%%%%%%%%%%%%%%%%%%%%%%%%%%%%%%%
%%%%%%%%%%%%%%%%%%%%%%%%%%%%%%%%%%%%%%%%%%%%%%%%%%%%%%%%%%%%%%%%%%%%%%%%%%%%%%%%
%%%%%%%%%%%%%%%%%%%%%%%%%%%%%%%%%%%%%%%%%%%%%%%%%%%%%%%%%%%%%%%%%%%%%%%%%%%%%%%%
%%%%%%%%%%%%%%%%%%%%%%%%%%%%%%%%%%%%%%%%%%%%%%%%%%%%%%%%%%%%%%%%%%%%%%%%%%%%%%%%

%%%%%%%%%%%%%%%%%%%%%%%%%%%%%%%%%%
%%%%%%%%% Appendice A %%%%%%%%%%%%
%%%%%%%%%%%%%%%%%%%%%%%%%%%%%%%%%%

\section{Ground states, action minimizers and bound states}
\label{app-gbstates}

In this section, we show that both $\delta$-NLS ground states and $\delta$-NLS action minimizers are $\delta$-NLS bound states, i.e. they satisfy \eqref{eq-regbound} and \eqref{EL3}.

First, we note that (using either the Lagrange Multipliers theorem in the former case or the simple Du Bois-Reymond equation in the latter case), if $u$ is either a $\delta$-NLS ground state of mass $\mu$ or a $\delta$-NLS action minimizers at frequency $\omega$, then it satisfies, for any fixed $\la>0$,
\begin{multline}
\label{ELdelta}
\langle\na\chi_{\la} ,\na\phi_{\la}\rangle+\la\langle\chi_{\la} ,\phi_{\la}\rangle+(\omega-\la) \langle \chi,u\rangle+\bar{\xi}q\left(\alpha+\theta_\la\right)-\langle\chi,|u|^{p-2}u\rangle=0\\
\forall\chi=\chi_{\la}+\xi\G_{\la}\in D.
\end{multline}
Whenever $u$ is a $\delta$-NLS ground state of mass $\mu$, $\omega=\mu^{-1}(\|u\|_p^p-Q(u))$. Now, letting $\xi=0$ in \eqref{ELdelta}, so that $\chi=\chi_\la\in H^1(\Rd)$, there results
\begin{equation}
\langle\na\chi ,\na\phi_{\la}\rangle+\langle\chi ,\omega\phi_{\la}+(\omega-\la)q\G_{\la}-|u|^{p-2}u\rangle=0\qquad\forall\chi\in H^{1}(\Rd).
\end{equation}
Hence, as $\omega\phi_{\la}+(\omega-\la)q\G_{\la}-|u|^{p-2}u\in L^2(\Rd)$, $\phi_{\la}\in H^{2}(\Rd)$ and, by density, 
\begin{equation}
\label{EL2}
-\lap\phi_{\la}+\omega\phi_{\la}+(\omega-\la)q\G_{\la}-|u|^{p-2}u=0\qquad\text{in}\quad L^2(\Rd), 
\end{equation}
which is equivalent to \eqref{EL3}. On the other hand, letting $\chi_{\la}=0$ and $\xi=1$ in \eqref{ELdelta}, so that $\chi=\G_{\la}$, there results
\begin{equation}
 \langle \G_{\la}, (\omega-\la)u-|u|^{p-2}u\rangle+q\left(\alpha+\theta_\la\right)=0.
\end{equation} 
Finally, using \eqref{EL2}, we obtain
\begin{equation}
 \langle \G_{\la}, (-\lap+\la)\phi_{\la}\rangle=q\left(\alpha+\theta_\la\right),
\end{equation}
which is equivalent to $\phi_{\la}(0)=q\left(\alpha+\theta_\la\right)$, so that also \eqref{eq-regbound} is satisfied.

%%%%%%%%%%%%%%%%%%%%%%%%%%%%%%%%%%%%%%%%%%%%%%%%%%%%%%%%%%%%%%%%%%%%%%%%%%%%%%%%
%%%%%%%%%%%%%%%%%%%%%%%%%%%%%%%%%%%%%%%%%%%%%%%%%%%%%%%%%%%%%%%%%%%%%%%%%%%%%%%%
%%%%%%%%%%%%%%%%%%%%%%%%%%%%%%%%%%%%%%%%%%%%%%%%%%%%%%%%%%%%%%%%%%%%%%%%%%%%%%%%
%%%%%%%%%%%%%%%%%%%%%%%%%%%%%%%%%%%%%%%%%%%%%%%%%%%%%%%%%%%%%%%%%%%%%%%%%%%%%%%%
%%%%%%%%%%%%%%%%%%%%%%%%%%%%%%%%%%%%%%%%%%%%%%%%%%%%%%%%%%%%%%%%%%%%%%%%%%%%%%%%
%%%%%%%%%%%%%%%%%%%%%%%%%%%%%%%%%%%%%%%%%%%%%%%%%%%%%%%%%%%%%%%%%%%%%%%%%%%%%%%%
%%%%%%%%%%%%%%%%%%%%%%%%%%%%%%%%%%%%%%%%%%%%%%%%%%%%%%%%%%%%%%%%%%%%%%%%%%%%%%%%
%%%%%%%%%%%%%%%%%%%%%%%%%%%%%%%%%%%%%%%%%%%%%%%%%%%%%%%%%%%%%%%%%%%%%%%%%%%%%%%%
%%%%%%%%%%%%%%%%%%%%%%%%%%%%%%%%%%%%%%%%%%%%%%%%%%%%%%%%%%%%%%%%%%%%%%%%%%%%%%%%
%%%%%%%%%%%%%%%%%%%%%%%%%%%%%%%%%%%%%%%%%%%%%%%%%%%%%%%%%%%%%%%%%%%%%%%%%%%%%%%%
%%%%%%%%%%%%%%%%%%%%%%%%%%%%%%%%%%%%%%%%%%%%%%%%%%%%%%%%%%%%%%%%%%%%%%%%%%%%%%%%
%%%%%%%%%%%%%%%%%%%%%%%%%%%%%%%%%%%%%%%%%%%%%%%%%%%%%%%%%%%%%%%%%%%%%%%%%%%%%%%%

%%%%%%%%%%%%%%%%%%%%%%%%%%%%%%%%%%
%%%%%%%%% Appendice B %%%%%%%%%%%%
%%%%%%%%%%%%%%%%%%%%%%%%%%%%%%%%%%

\section{Energy and action}
\label{app-enact}

\begin{proof}[Proof of Lemma \ref{chargs}]
Let $u$ be a $\delta$-NLS ground state at mass $\mu$ and let $\omega>0$ be the associated Lagrange multiplier, given by $\omega=\mu^{-1}(\|u\|_p^p-Q(u))$. Assume, by contradiction, that there exists $v=\eta_\la+\xi\G_\la \in N_\omega$ such that $S_\omega(v)<S_\omega(u)$ and let $\sigma>0$ be such that $\|\sigma v\|_2^2=\mu$. Then
\begin{equation*}
S_\omega(\sigma v)=\f{\sigma^2}{2}Q_\omega(v)-\f{\sigma^p}{p}\|v\|_p^p.
\end{equation*}
Computing the derivative with respect to $\sigma$ and using that $v\in N_\omega$, we get 
\begin{equation*}
\f{d}{d\sigma}S_\omega(\sigma v)=\sigma Q_\omega(v)-\sigma^{p-1}\|v\|_p^p=\sigma I_\omega(v)+(\sigma-\sigma^{p-1})\|v\|_p^p=\sigma(1-\sigma^{p-2})\|v\|_p^p,
\end{equation*}
which is greater than or equal to zero if and only if $0<\sigma\leq 1$. Hence $S_\omega(\sigma v)\le S_\omega(v)$, for every $\sigma>0$. Therefore, since $S_\omega(\sigma v)\le S_\omega(v) <S_\omega(u)$,
\begin{equation*}
E(\sigma v)+\f{\omega}{2}\|\sigma v\|_2^2<E(u)+\f{\omega}{2}\|u\|_2^2,
\end{equation*}
and using the fact that $\|\sigma v\|_2^2=\|u\|_2^2=\mu$, this entails $E(\sigma v)<E(u)$. However, as this contradicts the assumptions on $u$, we obtain that $u$ is a $\delta$-NLS action minimizer at frequency $\omega$.
\end{proof}

%%%%%%%%%%%%%%%%%%%%%%%%%%%%%%%%%%%%%%%%%%%%%%%%%%%%%%%%%%%%%%%%%%%%%%%%%%%%%%%%
%%%%%%%%%%%%%%%%%%%%%%%%%%%%%%%%%%%%%%%%%%%%%%%%%%%%%%%%%%%%%%%%%%%%%%%%%%%%%%%%
%%%%%%%%%%%%%%%%%%%%%%%%%%%%%%%%%%%%%%%%%%%%%%%%%%%%%%%%%%%%%%%%%%%%%%%%%%%%%%%%
%%%%%%%%%%%%%%%%%%%%%%%%%%%%%%%%%%%%%%%%%%%%%%%%%%%%%%%%%%%%%%%%%%%%%%%%%%%%%%%%
%%%%%%%%%%%%%%%%%%%%%%%%%%%%%%%%%%%%%%%%%%%%%%%%%%%%%%%%%%%%%%%%%%%%%%%%%%%%%%%%
%%%%%%%%%%%%%%%%%%%%%%%%%%%%%%%%%%%%%%%%%%%%%%%%%%%%%%%%%%%%%%%%%%%%%%%%%%%%%%%%
%%%%%%%%%%%%%%%%%%%%%%%%%%%%%%%%%%%%%%%%%%%%%%%%%%%%%%%%%%%%%%%%%%%%%%%%%%%%%%%%
%%%%%%%%%%%%%%%%%%%%%%%%%%%%%%%%%%%%%%%%%%%%%%%%%%%%%%%%%%%%%%%%%%%%%%%%%%%%%%%%
%%%%%%%%%%%%%%%%%%%%%%%%%%%%%%%%%%%%%%%%%%%%%%%%%%%%%%%%%%%%%%%%%%%%%%%%%%%%%%%%
%%%%%%%%%%%%%%%%%%%%%%%%%%%%%%%%%%%%%%%%%%%%%%%%%%%%%%%%%%%%%%%%%%%%%%%%%%%%%%%%
%%%%%%%%%%%%%%%%%%%%%%%%%%%%%%%%%%%%%%%%%%%%%%%%%%%%%%%%%%%%%%%%%%%%%%%%%%%%%%%%
%%%%%%%%%%%%%%%%%%%%%%%%%%%%%%%%%%%%%%%%%%%%%%%%%%%%%%%%%%%%%%%%%%%%%%%%%%%%%%%%

%%%%%%%%%%%%%%%%%%%%%%%%%%%%%%%%%%
%%%%%%%%% Bibliografia %%%%%%%%%%%
%%%%%%%%%%%%%%%%%%%%%%%%%%%%%%%%%%

\section{Stability of the set of ground states}
\label{app-stab}

In this section, we show that the set of ground states at mass $\mu$, denoted by $\mathcal{A}_{\mu}$, is orbitally stable. Although this is an expected result, we report it here for the sake of completeness. The proof is obtained adaptating the arguments in \cite{CL-82} and collecting some other results already present in the literature. 

Fix $\la>0$. Let us recall that the energy domain \eqref{dom} can be endowed with the natural norm
\begin{equation}
\label{Dnorm}
\|\psi\|_{D}:=\left(\|\na \phi_{\la}\|_{2}^{2}+\la \|\phi_{\la}\|_{2}^{2}+(\alpha+\theta_{\la})|q|^{2}\right)^{\f{1}{2}},
\end{equation}
and denote by $D^{*}$ the dual space of $D$. In view of \eqref{Dnorm}, the expression of the energy $E$ in \eqref{E} can be written as 
\begin{equation*}
E(\psi)=\f{1}{2}\|\psi\|_{D}^{2}-\f{\la}{2}\|\psi\|_{2}^{2}-\f{1}{p}\|\psi\|_{p}^{p}.
\end{equation*}

Let us then consider the Cauchy problem
\begin{equation}
\label{CP}
\begin{cases}
i\f{\partial \psi}{\partial t}=H_\alpha \psi-|\psi|^{p-2}\psi\\ 
\psi(0)=\psi_{0},
\end{cases}
\end{equation}
and define its weak solutions as follows.
\begin{definition}
Let $I$ be an open interval such that $0\in I\subset \R$. A function $\psi\in L^{\infty}(I;D)$ is called a local weak solution to \eqref{CP} on $I$ if $\psi$ belongs to $L^{\infty}(I;D)\cap W^{1,\infty}(I;D^{*})$ and satisfies \eqref{CP} in the sense of $L^{\infty} (I ; D^{*})$. In particular, if $I$ coincides with $\R$, then $\psi$ is called a global weak solution to \eqref{CP}.
\end{definition}

The next result concerns the global well-posedness in $D$ and is the first ingredient to prove the orbital stability of $\mathcal{A}_{\mu}$ via \cite{CL-82}. The proof is obtained by combining inequality \cite[eq. (2.11)]{CFN-21} and the results about the local well-posedness obtained in \cite[Appendix B]{FGI-21}.

\begin{proposition}[Global well-posedness in $D$]
\label{GWP}
Let $2<p<4$. Then, for any $\psi_{0}\in D$ there exists a unique global weak solution 
\begin{equation*}
\psi\in C(\R; D)\cap C^{1}(\R; D^{*})
\end{equation*} 
of \eqref{CP}. Moreover, the following conservation laws hold:
\begin{gather}
 \label{cons-mass} \|\psi(t)\|_{L^{2}(\R^{2})}=\|\psi_{0}\|_{L^{2}(\R^{2})},\qquad\forall\,t\in \R ,\\[.2cm]
 \label{cons-en} E(\psi(t))=E(\psi_{0}),\qquad \forall\,t\in \R.
\end{gather}
\end{proposition}
\begin{proof}
%Let $\psi_{0}\in D$. By \cite[Proposition 1.1]{FGI-21}, there exists the unique maximal solution 
%\begin{equation*}
%\psi\in C((-T_{\min}, T_{\max}); D)\cap C^{1}((-T_{\min}, T_{\max}); D^{*})
%\end{equation*} 
%of \eqref{CP} and the conservation laws hold:
%\begin{gather}
% \label{cons-mass} \|\psi(t)\|_{L^{2}(\R^{2})}=\|\psi_{0}\|_{L^{2}(\R^{2})}\quad\forall\,t\in (-T_{\min}, T_{\max}) ,\\
% \label{cons-en} E(\psi(t))=E(\psi_{0})\quad \forall\,t\in (-T_{\min}, T_{\max}).
%\end{gather}
%In order to extend the maximal solution $\psi$ on $\R$, we need to prove a priori estimate. In particular, by \cite[Lemma B.1]{FGI-21} 
The proof is an application of \cite[Theorem 2.4]{OSY-12}, that deals with abstract NLSE in the spirit of \cite{C-CL03}, but with general self-adjoint operators in the place of the standard Laplacian. The hypothesis to be verified are the six conditions \cite[\bf{(G1)--(G6)}]{OSY-12} on the nonlinear term $g(\psi)=-|\psi|^{p-2}\psi$ of the equation, together with a uniqueness result for the solutions to \eqref{CP}. The first five conditions {\bf(G1)--(G5)} and the uniqueness result are proved respectively in \cite[Lemma B.1]{FGI-21} and \cite[Lemma B.2]{FGI-21} and are sufficient for the local well-posedness in $D$. We are left to prove hypothesis {\bf (G6)}, that reads in our context as follows:
\begin{equation*}
\begin{split}
{\bf (G6)}\quad &\text{$\exists\,\ep\in(0,1]$ and $C_{0}(\cdot)\geq 0$ }: \, \f{1}{p}\|\psi\|_{p}^{p}\leq \f{1-\ep}{2}\|\psi\|_{D}^{2}+C_{0}(\|\psi\|_{2}),\qquad\forall \,\psi \in D.
\end{split}
\end{equation*} 
However, by using \cite[eq. (2.11)]{CFN-21} and inequality $ab\leq \ep a^{r}+C(\ep)b^{\f{r}{r-1}}$, with $r=\f{2}{p-2}$, there results
\begin{equation*}
\f{1}{p}\|\psi\|_{p}^{p}\leq \f{C_{p}}{p}\|\psi\|_{D}^{p-2}\|\psi\|_{2}^{2}\leq \f{\ep C_{p}}{p}\|\psi\|_{D}^{2}+\f{C(\ep)C_{p}}{p}\|\psi\|_{2}^{\f{4}{4-p}}, 
\end{equation*}
which proves {\bf (G6)} and concludes the proof.
\end{proof}

Now, we can introduce the definition of stability and prove the aimed result.

\begin{definition}
Fix $\mu>0$. We say that the set of ground states $\mathcal{A}_{\mu}$ is orbitally stable if for any $\ep>0$ there exists $\delta>0$ such that for any $\psi_{0}\in D$ satisfying $\inf_{u\in \mathcal{A}_{\mu}}\|\psi_{0}-u\|_{D}<\delta$, the unique global solution $\psi(t)$ of \eqref{CP} satisfies $\inf_{u\in \mathcal{A}_{\mu}}\|\psi(t)-u\|_{D}<\ep$ for any $t\in \R$.
\end{definition}

\begin{proposition}
\label{stabGmu}
For any $\mu>0$ the set of ground states $\mathcal{A}_{\mu}$ is orbitally stable.
\end{proposition}
\begin{proof}
We prove it by contradiction as in \cite{CL-82}. Suppose that $\mathcal{A}_{\mu}$ is not orbitally stable. This means that there exists $\ep_{0}>0$, a sequence $(\psi_{0}^{n})_{n}\subset D$ and a sequence $(t_{n})_{n}\subset \R$  such that 
\begin{equation}
\label{inf0}
\inf_{u\in \mathcal{A}_{\mu}}\|\psi_{0}^{n}-u\|_{D}\to 0, \quad\text{as}\quad n\to +\infty,
\end{equation}
but 
\begin{equation}
\label{contr}
\inf_{u\in \mathcal{A}_{\mu}}\|\psi^{n}(t_{n})-u\|_{D}\geq \ep_{0}, \quad\text{for every}\quad n\in \N,
\end{equation}
where $\psi^{n}$ is the unique global solution of \eqref{CP} with initial datum $\psi_{0}^{n}$ provided by Proposition \ref{GWP}.

The convergence in \eqref{inf0} entails the existence of a sequence $(u_{n})_{n}\subset \mathcal{A}_{\mu}$ such that $\|\psi_{0}^{n}-u_{n}\|_{D}\to 0$ as $n\to +\infty$. It is straightforward to check that $\|\psi_{0}^{n}\|_{2}^{2}\to \mu$ as $n\to +\infty$. Moreover, being $(u_{n})\subset \mathcal{A}_{\mu}$, they satisfy $E(u_{n})=\Eps(\mu)<0$ and, applying \cite[eq. (2.11)]{CFN-21}, it turns out that $\|u_{n}\|_{D}$ and $\|u_{n}\|_{p}$ are bounded. As a consequence, since $\|\psi_{0}^{n}\|_{D}\le \|\psi_{0}^{n}-u_{n}\|_{D}+\|u_{n}\|_{D}$, the boundedness of $\|\psi_{0}^{n}\|_{D}$ follows. Moreover, the same argument together with \cite[eq. (2.11)]{CFN-21} can be used to prove the boundedness of $\|\psi_{0}^{n}\|_{p}$. By using these estimates, one can show that $E(\psi_{0}^{n})\to \Eps(\mu)$. Indeed, 
\begin{equation*}
\begin{split}
E(\psi_{0}^{n})-E(u_{n})&\leq \f{1}{2}\left|\|\psi_{0}^{n}\|_{D}^{2}-\|u_{n}\|_{D}^{2}\right|+\f{\la}{2}\left|\|\psi_{0}^{n}\|_{2}^{2}-\|u_{n}\|_{2}^{2}\right|+\f{1}{p}\left|\|\psi_{0}^{n}\|_{p}^{p}-\|u_{n}\|_{p}^{p}\right|\\
&\leq \left(\|\psi_{0}^{n}\|_{D}+\|u_{n}\|_{D}\right)\|\psi_{0}^{n}-u_{n}\|_{D}^{2}+ \left(\|\psi_{0}^{n}\|_{2}+\|u_{n}\|_{2}\right)\|\psi_{0}^{n}-u_{n}\|_{2}\\
&+\max\{\|\psi_{0}^{n}\|_{p}^{p-1},\|u_{n}\|_{p}^{p-1}\}\|\psi_{0}^{n}-u_{n}\|_{p}\to 0 \quad\text{as}\quad n\to+\infty.
\end{split}
\end{equation*}

In view of \eqref{cons-mass} and \eqref{cons-en}, we have that 
\begin{equation}
\label{doubleconv}
\|\psi^{n}(t_{n})\|_{2}^{2}\to \mu\quad\text{and}\quad E(\psi^{n}(t_{n}))\to \Eps(\mu)\quad\text{as}\quad n\to +\infty.
\end{equation} 
By \eqref{doubleconv} and \cite[eq. (2.11)]{CFN-21}, both $\|\psi^{n}(t_{n})\|_{D}$ and $\|\psi^{n}(t_{n})\|_{p}$ are bounded. Moreover, if we define $\xi_{n}:=\f{\sqrt{\mu}}{\|\psi^{n}(t_{n})\|_{2}}\psi^{n}(t_{n})$, then $\|\xi_{n}\|_{2}^{2}=\mu$ and, by using $\|\psi^{n}(t_{n})\|_{2}^{2}\to \mu$ and the boundedness of $\|\psi^{n}(t_{n})\|_{D}$ and $\|\psi^{n}(t_{n})\|_{p}$,  
\begin{equation*}
E(\xi_{n})=E(\psi^{n}(t_{n}))+o(1),\quad \text{as}\quad n\to+\infty.
\end{equation*}
This entails that $\xi_{n}$ is a minimizing sequence for $E$ of mass $\mu$. Hence, arguing as in the proof of Theorem \ref{exchar-gs}, one has that there exists $u\in \mathcal{A}_{\mu}$ such that $\|\xi_{n}-u\|_{D}\to 0$ as $n\to+\infty$. By the definition of $\xi_{n}$ and the facts that $\|\psi^{n}(t_{n})\|_{2}^{2}\to \mu$ as $n\to+\infty$ and $\|\psi^{n}(t_{n})\|_{D}$ is bounded, there results that 
\begin{equation*}
\|\psi^{n}(t_{n})-u\|_{D}\to 0\quad\text{as}\quad n\to+\infty,
\end{equation*}
being in contradiction with \eqref{contr}.
\begin{remark}
Proposition \ref{stabGmu} deals with the orbital stability of the whole set of ground states $\mathcal{A}_{\mu}$. As explained in \cite{CL-82}, a natural improvement of such a result is the orbital stability of a single ground state (up to gauge invariace), which is a straightforward consequence of Proposition \ref{stabGmu} as soon as one can prove the uniqueness of the ground state (up to gauge invariace). This is true, for instance, for the standard $L^{2}$-subcritical NLSE \cite{K-89} and could be an interesting topic to be studied in the context of the $\delta-$NLSE in a forthcoming paper.
\end{remark}
\end{proof}

\end{document}